\documentclass[a4paper, 10pt]{extarticle}
\usepackage[top=70pt,bottom=70pt,left=75pt,right=77pt]{geometry}

\usepackage{amssymb, amsmath, amsthm, setspace, bbm, prettyref, graphicx, float, subfigure, color, mathrsfs, mathtools}

\usepackage{Lutsko_Style}
\DeclareMathSizes{10}{9}{7}{5}
\usepackage[normalem]{ulem}
\usepackage{soul}

\usepackage{graphicx}

\global\long\def\vep{\varepsilon}

\title{
Full Poissonian Local Statistics of Slowly Growing Sequences
}
 \author{
{\sc Christopher Lutsko
and Niclas Technau
}
}
\date{}

\begin{document}

\maketitle

  \setlength{\abovedisplayskip}{1mm}
  
\begin{abstract}
  Fix $\alpha>0$, then by Fejér's theorem $ (\alpha(\log n)^{A}\,\mathrm{mod}\,1)_{n\geq1}$ is uniformly distributed if and only if $A>1$. We sharpen this by showing that all correlation functions, and hence the gap distribution, are Poissonian provided $A>1$. This is the first example of a deterministic sequence modulo one whose gap distribution, and all of whose correlations are proven to be Poissonian. The range of $A$ is optimal and complements a result of Marklof and Str\"{o}mbergsson who found the limiting gap distribution of $(\log(n)\, \mathrm{mod}\,1)$, which is necessarily not Poissonian. 
\end{abstract}

  \section{Introduction}

  A sequence $(x(n))_{n\geq1}\subseteq[0,1)$ is \emph{uniformly distributed modulo $1$} if the proportion of points in any subinterval $I \subseteq[0,1)$ converges to the size of the interval: $\#\{n\leq N:\,x(n)\in I \}\sim N|I| $, as $N\rightarrow\infty$. The theory of uniform distribution dates back to 1916, to a seminal paper of Weyl \cite{Weyl1916}, and constitutes a simple test of pseudo-randomness.  A well-known result of Fej\'{e}r, see \cite[p. 13]{KupiersNiederreiter1974}, implies that for any $A>1$ and any $\alpha>0$ the sequence
\[
(\alpha(\log n)^{A}
\mathrm{\,mod\,}1)_{n>0}
\]
is uniformly distributed. While for $A=1$, the sequence is \emph{not uniformly distributed}.  In this paper, we study stronger, local tests for pseudo-randomness for this sequence.

Given an increasing $\R$-valued sequence, $(\omega(n))=(\omega(n))_{n>0}$ we denote the sequence modulo $1$ by
  \begin{align*}
    x(n) : = \omega(n) \Mod 1.
  \end{align*}
  Further, let $u_N(n)\subset [0,1)$ denote the ordered sequence, thus $u_N(1) \le u_N(2) \le \dots \le u_N(N)$. With that, we define the \emph{gap distribution} of $(x(n))$ as the limiting distribution (if it exists): for $s>0$
  \begin{align*}
    P(s) : = \lim_{N \to \infty}\frac{ \#\{n \le N : N\|u_N(n)-u_{N}(n+1)\| < s\}}{N},
  \end{align*}
  where $\|\cdot \|$ denotes the distance to the nearest integer, and $u_N(N+1)= u_N(1)$. Thus $P(s)$ represents the limiting proportion of (scaled) gaps between (spatially) neighboring elements in the sequence which are less than $s$. We say a sequence has \emph{Poissonian gap distribution} if $P(s) =
  1- e^{-s}$, the expected value for a uniformly distributed sequence on the unit interval.

  \begin{figure}[H]
\includegraphics[viewport=50bp 460bp 330bp 770bp,clip,scale=0.5]{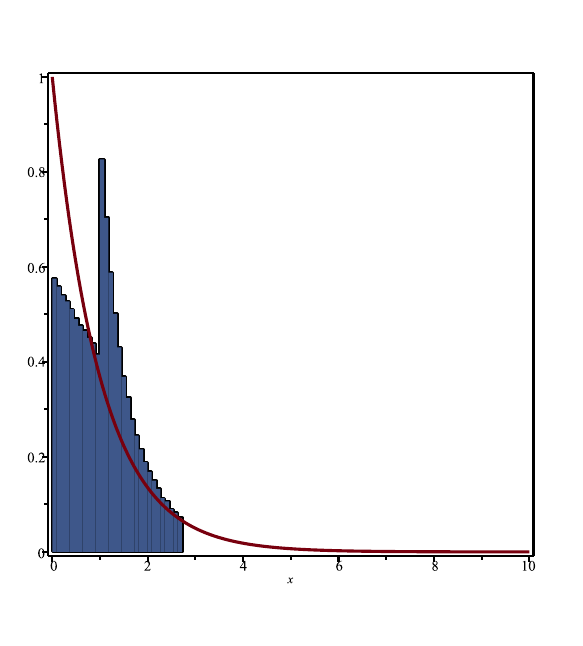}~~~~\includegraphics[viewport=50bp 460bp 330bp 770bp,clip,scale=0.5]{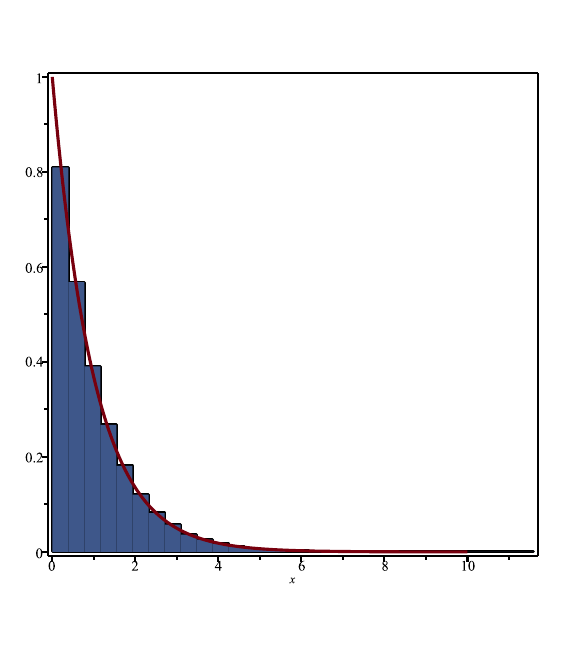}~~~\includegraphics[viewport=50bp 460bp 330bp 770bp,clip,scale=0.5]{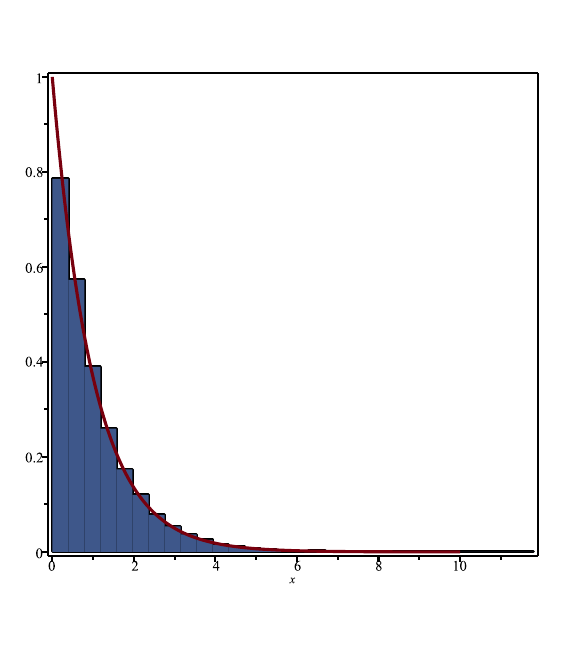}

\caption{From left to right: the histograms represent the gap distribution {\emph{density}} at time $N$ of $(\log n)_{n\protect\geq1}$, $((\log n)^{2})_{n\protect > 0}$, and $((\log n)^{3})_{n\protect> 0}$ when $N=10^{5}$ and the curve is the graph of $x\protect\mapsto e^{-x}$. Note that $(\log n)$ is not even uniformly distributed, and thus the gap distribution cannot be Poissonian.}
\end{figure}  

  \newpage
  \noindent Our main theorem is the following
  \begin{theorem}\label{thm:main}
    Let $\omega(n):= \alpha ( \log n)^{A}$ for $A>1$ and any $\alpha >0$, then $x(n)$ has Poissonian gap distribution.
  \end{theorem}
  \vspace{1mm}

  In fact, this theorem follows (via the method of moments) from Theorem \ref{thm:correlations} (below) which states that for every $m \ge 2$ the $m$-point correlation function of this sequence is Poissonian. By which we mean the following: Let $m\geq 2$ be an integer, and let $f \in C_c^\infty(\R^{m-1})$ be a compactly supported function which can be thought of as a stand-in for the characteristic function of a Cartesian product of compact intervals in $\R^{m-1}$. Let $[N]:=\{1,\ldots,N\}$ and define the \emph{$m$-point correlation} of $(x(n))$, at time $N$, to be
  \begin{equation}\label{def: m-point correlation function}
    R^{(m)}(N,f) := \frac{1}{N} 
    \sum_{\vect{n} \in [N]^m}^\ast 
    f(N\|x(n_1)-x(n_2)\|,N \|x(n_2)-x(n_3)\|, \dots, N\|x(n_{m-1})-x(n_{m})\|),
  \end{equation}
  where $\ds \sum^\ast$ denotes a sum over distinct $m$-tuples. Thus the $m$-point correlation measures how correlated points are on the scale of the average gap between neighboring points (which is $N^{-1}$). We say $(x(n))$ has \emph{Poissonian $m$-point correlation} if
  \begin{align}\label{def: expectation}
    \lim_{N \to \infty} R^{(m)}(N,f) = \int_{\R^{m-1}} f(\vect{x}) \mathrm{d}\vect{x} =: \expect{f} 
    \,\, \mathrm{for\,any\,} f \in C_c^\infty(\R^{m-1}).
  \end{align}
  That is, if the $m$-point correlation converges to the expected value if the sequence was uniformly distributed on the unit interval.

  \begin{theorem}\label{thm:correlations}
    Let $\omega(n):= \alpha (\log n)^{A}$ for $A>1$ and any $\alpha >0$, then $x(n)$ has Poissonian $m$-level correlations for all $m \ge 2$.
  \end{theorem}
  \vspace{1mm}

  It should be noted that Theorem \ref{thm:correlations} is far stronger than Theorem \ref{thm:main}. In addition to the gap distribution, Theorem \ref{thm:correlations} allows us to recover a wide-variety of statistics such as the $i^{th}$ nearest neighbor distribution for any $i \ge 1$.

\vspace{2mm}
  
\textbf{Previous Work:} The study of uniform distribution and fine-scale local statistics of sequences modulo $1$ has a long history which we outlined in more detail in a previous paper \cite{LutskoSourmelidisTechnau2021}. If we consider the sequence $(\alpha n^\theta\,\mathrm{mod}\,1)_{n\geq1}$, there have been many attempts to understand the local statistics, in particular the pair correlation (when $m=2$). Here it is known that for any $\theta \neq 1$, then, if $\alpha$ belongs to a set of full measure, the pair correlation function is Poissonian \cite{RudnickSarnak1998, AistleitnerEl-BazMunsch2021, RudnickTechnau2021}. However there are very few explicit (i.e. non-metric) results. When $\theta = 2$ Heath-Brown \cite{Heath-Brown2010} gave an algorithmic construction of certain $\alpha$ for which the pair correlation is Poissonian, however this construction did not give an exact number. When $\theta = 1/2$ and $\alpha^2 \in \Q$ the problem lends itself to tools from homogeneous dynamics. This was exploited by Elkies and McMullen \cite{ElkiesMcMullen2004} who showed that the gap distribution is \emph{not} Poissonian, and by El-Baz, Marklof, Vinogradov \cite{El-BazMarklofVinogradov2015a} who showed that the sequence $(\alpha n^{1/2}\, \mathrm{mod}\,1)_{n\in \N \setminus \square}$ where $\square$ denotes the set of squares,  does have Poissonian pair correlation.

With these sparse exceptions, the only explicit results occur when the exponent $\theta$ is small. If $\theta\le 14/41$ the authors and Sourmelidis \cite{LutskoSourmelidisTechnau2021} showed that the pair correlation function is Poissonian for all values of $\alpha >0$. This was later extended by the authors \cite{LutskoTechnau2021} to show that these monomial sequences exhibit Poissonian $m$-point correlations (for $m\geq 3$) for any $\alpha >0$ if $\theta< 1/(m^2+m-1)$. To the best of our knowledge the former is the only explicit result proving Poissonian pair correlations for sequences modulo $1$, and the latter result is the only result proving convergence of the higher order correlations to any limit.

The authors' previous work motivates the natural question: what about sequences which grow slower than any power of $n$? It is natural to hypothesize that such sequences might exhibit Poissonian $m$-point correlations for all $m$. However, there is a constraint, Marklof and Str\"{o}mbergsson \cite{MarklofStrom2013} have shown that the gap distribution of $( (\log n)/(\log b) \,\mathrm{mod}\, 1)_{n\geq 1}$ exists for $b>1$, and is \emph{not} Poissonian (thus the correlations cannot all be Poissonian). However, they also showed, that in the limit as $b$ tends to $1$, this limiting distribution converges to the Poissonian distribution (see \cite[(74)]{MarklofStrom2013}). Thus, the natural quesstion becomes: what can be said about sequences growing faster than $\log(n)$ but slower than any power of $n$?

With that context in mind, our result has several implications. First, it provides the only example at present of an explicit sequence whose $m$-point correlations can be shown to converge to the Poissonian limit (and thus whose gap distribution is Poissonian). Second, it answers the natural question implied by our previous work on monomial sequences. Finally, it answers the natural question implied by Marklof and Str\"{o}mbergsson's result on logarithmic sequences.

  \subsection{Plan of Paper}

  The proof of Theorem \ref{thm:correlations}
  adds several new ideas to the method developed 
  in \cite{LutskoTechnau2021}, which
  is insufficient for the definitive results established here.
  Broadly we argue in three steps,
  detailing the difficulties and innovations in each step.

  In the remainder we take $\alpha = 1$, 
  the same exact proof applies to general $\alpha$ 
  leaving straightforward adaptations aside. 
  Fix $m \ge 2$ and assume the sequence has 
  Poissonian $j$-point correlation for $2\le j< m$.

  \begin{enumerate}[label = {\tt [Step \arabic*]}]
  \item Remove the distinctness condition in the $m$-point correlation by relating the completed correlation to the $m^{th}$ moment of a random variable. This will add a new frequency variable, with the benefit of decorrelating the sequence elements. Then we perform a Fourier decomposition of this moment and using a combinatorial argument from \cite[\S 3]{LutskoTechnau2021}, we reduce the problem of convergence for the moment to convergence of one particular term to an explicit `target'.
  This step works quite similar to what 
  we did in \cite{LutskoTechnau2021}.
  
  \item Using various partitions of unity we further reduce the problem to an asymptotic evaluation of the $L^m([0,1])$-norm of a two dimensional exponential sum. Then we apply van der Corput's $B$-process in each of these variables. In contrast to our argument in \cite{LutskoTechnau2021},
  we can not longer use the form of the sequence 
  to perform explicit computations throughout. 
  Instead a more fundamental understanding of
  how the two $B$-process work is now required.
  In fact, after the first application of the $B$-process
  we end up with an implicitly defined phase function. 
  Surprisingly, after the second application 
  of the $B$-process (in the other variable)
  that we can show that a manageable phase function arises!
  This is the content of Lemma \ref{lem:Phi mu},
  and we believe this by-product of 
  our investigation to be of some independent interest.
  Being able to understand the arising phase function 
  is crucial to perform the next step.
  Further, a simple computation yields that if we stop at this step 
  and apply the triangle inequality the resulting error term 
  is of size $O((\log N)^{(A+1)m})$.
  
  \item Finally we expand the $L^m([0,1])$-norm giving an oscillatory integral. Then using a localized version of 
  Van der Corput's lemma we achieve 
  an extra saving to bound the error term by $o(1)$. 
  In \cite{LutskoTechnau2021} we used classical theorems from linear algebra to justify that that localized version 
  of Van der Corput's lemma is applicable, by 
  showing that Wronskians of a family of
  relevant curves is uniformly bounded from below.
  In the present situation, the underlying geometry
  and Wronskians are considerably more involved. 
  After several initial manipulations we boil matters
  down to determinants of generalized Vandermonde matrices.
  To handle those we rely on recent work of 
  Khare and Tao \cite{KhareTao2021}, which is precise 
  enough so that can barely (by some logarithmic gains)
  single out the largest contribution to the Wronskian 
  and thereby complete the argument.
  \end{enumerate}

\textbf{Notation:}
Throughout, we use the usual Bachmann--Landau notation: for functions $f,g:X \rightarrow \mathbb{R}$, defined on some set $X$, we write $f \ll g$ (or $f=O(g)$) to denote that there exists a constant $C>0$ such that $\vert f(x)\vert \leq C \vert g(x) \vert$ for all $x\in X$. Moreover let $f\asymp g$ denote $f \ll g$ and $g \ll f$, and let $f = o(g)$ denote that $\frac{f(x)}{g(x)} \to 0$.

Given a Schwartz function $f: \R^m \to \R$, let $\wh{f}$ denote the $m$-dimensional Fourier transform:
\begin{align*}
  \wh{f}(\vect{k}) : = \int_{\R^m} f(\vect{x}) e(-\vect{x}\cdot \vect{k}) \mathrm{d}\vect{x}, \qquad \text{for } \vect{k} \in \Z^m.
\end{align*}
Here, and throughout we let $e(x):= e^{2\pi i x}$.

All of the sums which appear range over integers, in the indicated interval. We will frequently be taking sums over multiple variables, thus if $\vect{u}$ is an $m$-dimensional vector, for brevity, we write
\begin{align*}
  \sum_{\vect{k} \in [f(\vect{u}),g(\vect{u}))}F(\vect{k}) = \sum_{k_1 \in [f(u_1),g(u_1))} \dots \sum_{k_m \in [f(u_m),g(u_m))}F(\vect{k}).
\end{align*}
Moreover, all $L^p$ norms are taken with respect to Lebesgue measure, we often do not include the domain when it is obvious. Let
\begin{align*}
  \Z^\ast:= \Z \setminus \{0\}. 
\end{align*}   
  For ease of notation, $\varepsilon>0$ may vary from line to line by a bounded constant. 

\section{Preliminaries}
\label{ss:The A and B Processes}

The following stationary phase principle is derived from the work of Blomer, Khan and Young \cite[Proposition 8.2]{BlomerKhanYoung2013}. In application we will not make use of the full asymptotic expansion, but this will give us a particularly good error term which is essential to our argument. 

\begin{proposition}\label{prop: stationary phase}[Stationary phase expansion]
  Let $w\in C_{c}^{\infty}$ be supported in a compact interval $J$ of length $\Omega_w>0$ so that there exists an $\Lambda_w>0$ for which
  \begin{align*}
    w^{(j)}(x)\ll_{j}\Lambda_w \Omega_w^{-j}
  \end{align*}
  for all $j\in\mathbb{N}$. Suppose $\psi$ is a smooth function on $J$ so that there exists a unique critical point $x_{0}$ with $\psi'(x_{0})=0$. Suppose there exist values $\Lambda_{\psi}>0$ and $\Omega_\psi>0$ such that 
  \begin{align*}
    \psi''(x)\gg \Lambda_{\psi}\Omega_{\psi}^{-2},\qquad \qquad\psi^{(j)}(x)\ll_j \Lambda_{\psi}\Omega_{\psi}^{-j}
  \end{align*}
  for all $j > 2$. Moreover, let $\delta\in(0,1/10)$, and $Z:=\Omega_w + \Omega_{\psi}+\Lambda_w+\Lambda_{\psi}+1$. If
  \begin{equation}
    \Lambda_{\psi}\geq Z^{3\delta}, \qquad \mathrm{and}\qquad \Omega_w\geq\frac{\Omega_{\psi}Z^{\frac{\delta}{2}}}{\Lambda_{\psi}^{1/2}}\label{eq: constraints on X,Y,Z}
  \end{equation}
  hold, then 
  \[
  I:=\int_{-\infty}^{\infty}w(x)e(\psi(x))\,\mathrm{d}x
  \]
  has the asymptotic expansion
  \[
  I=\frac{e(\psi(x_{0}))}{\sqrt{\psi''(x_{0})}}\sum_{0\leq j \leq3C/\delta}p_{j}(x_{0})+O_{C,\delta}(Z^{-C})
  \]
  for any fixed $C\in\mathbb{Z}_{\geq1}$; here
  \begin{align*}
    p_n(x_0) := \frac{e(1/8)}{n!} \left(\frac{i}{2 \psi^{\prime\prime}(x_0)}\right)^nG^{(2n)}(x_0)
  \end{align*}
  where
  \begin{align*}
    G(x):= w(x)e(H(x)), \qquad H(x) := \psi(x)-\psi(x_0) -\frac{1}{2} \psi^{\prime\prime}(x_0)(x-x_0)^2.
  \end{align*}

\end{proposition}

In a slightly simpler form we have:

\begin{lemma}[First order stationary phase] \label{lem: stationary phase}
  Let $\psi$ and $w$ be smooth, real valued functions defined on a compact interval $[a, b]$. Let $w(a) = w(b) = 0$. Suppose there exist constants $\Lambda_\psi,\Omega_w,\Omega_\psi \geq 3$ satisfying \eqref{eq: constraints on X,Y,Z}, with $Z$ as in Proposition \ref{prop: stationary phase} and $\Lambda_w=1$ so that
  \begin{align} \label{eq: growth conditions}
    \psi^{(j)}(x) \ll \frac{\Lambda_\psi}{ \Omega_\psi^j}, \ \ w^{(j)} (x)
    \ll \frac{1}{\Omega_w^j}\, \ \ \ 
    \text{and} \ \ \ \psi^{(2)} (x)\gg \frac{\Lambda_\psi}{ \Omega_\psi^2}
  \end{align} 
  for all $j=0,1,2,\dots$ and all $x\in [a,b]$. If $\psi^\prime(x_0)=0$ for a unique $x_0\in[a,b]$,  and if $\psi^{(2)}(x)>0$, then
  \begin{equation*}
    \begin{split}
      \int_{a}^{b} w(x) e(\psi(x))  \rd x = &\frac{e(\psi(x_0) + 1/8)}{\sqrt{\abs{\psi^{(2)}(x_0)}}}
      w(x_0)
      + O\left( \frac{\Omega_\psi^3} {\Lambda_{\psi}^{3/2}\Omega_w^2 } + \frac{1}{Z}
      \right).
    \end{split}
  \end{equation*}
  If instead $\psi^{(2)}(x)<0$ on $[a,b]$ then the same equation holds with $e(\psi(x_0) + 1/8)$ replaced by $e(\psi(x_0) - 1/8)$.
\end{lemma}
\begin{proof}
   We apply Proposition \ref{prop: stationary phase} with $\Lambda_w=1$ and $C=1$. In which case the first error term comes from the term $p_1$ in the expansion. All higher order terms give a smaller contribution, see
   \cite[p. 20]{BlomerKhanYoung2013} for a more detailed explanation.
\end{proof}

\noindent Moreover, we also need the following version of van der Corput's lemma (\cite[Ch. VIII, Prop. 2]{Stein1993}).

\begin{lemma}[van der Corput's lemma] \label{lem: van der Corput's lemma}
    Let $[a,b]$ be a compact interval.
    Let $\psi,w:[a,b]\rightarrow\mathbb{R}$ be smooth functions. 
    Assume $\psi''$ does not change sign
    on $[a,b]$ and that for some $j\geq 1$ and $\Lambda>0$ the bound 
    \[
    \vert\psi^{(j)}(x)\vert\geq\Lambda
    \]
    holds for all $x\in [a,b]$.
    Then
    \[
    \int_{a}^{b}\,e(\psi(x)) w(x)
    \,\mathrm{d}x \ll \Big(\vert w(b) \vert + \int_{a}^b \vert 
    w'(x)\vert\, \mathrm{d}x \Big) 
    \Lambda^{-1/j}
    \]
    where the implied constant depends only on $j$.
\end{lemma}

\textbf{Generalized Vandermonde matrices:} One of the primary difficulties which we will encounter in Section \ref{s:Off-diagonal} is the fact that taking derivatives of exponentials (which arise as the inverse of the $\log$'s defining our sequence) increases in complexity as we take derivatives. To handle this we appeal to a recent result of Khare and Tao \cite{KhareTao2021} which requires some notational set-up. Given an $M$-tuple $\vect{u} \in \R^M$, let
\begin{align*}
  V(\vect{u}) := \prod_{1\le i < j \le M}(u_j-u_i)
\end{align*}
denote the Vandermonde determinant. Furthermore given two tuples $\vect{u}$ and $\vect{n}$ we define
\begin{align*}
  \vect{u}^{\vect{n}} := u_1^{n_1} \cdots u_M^{n_M}, \qquad \mbox{ and } \qquad \vect{u}^{\circ\vect{n}} := \begin{pmatrix}
    u_1^{n_1} & u_1^{n_2} & \dots & u_1^{n_M}\\ 
    u_2^{n_1} & u_2^{n_2} & \dots & u_2^{n_M}\\ 
    \vdots &\vdots  & \vdots &\vdots \\  
    u_M^{n_1} & u_M^{n_2} & \dots & u_M^{n_M}
  \end{pmatrix},
\end{align*}
the latter being a generalized Vandermonde matrix. Finally let $\vect{n}_{\text{min}} :=(0,1, \dots, M-1)$. Then Khare and Tao established the following

\begin{lemma}[{\cite[Lemma 5.3]{KhareTao2021}}] \label{lem:KT}
  Let $K$ be a compact subset of the cone
  \begin{align*}
    \{(n_1, \dots, n_{M})\in \R^M \ : \ n_1 < \dots < n_{M} \}.
  \end{align*}
  Then there exist constants $C,c>0$ such that
  \begin{align}\label{KT}
    cV(\vect{u}) \vect{u}^{\vect{n}-\vect{n}_{\text{min}}} \le \det(\vect{u}^{\circ \vect{n}}) \le C V(\vect{u}) \vect{u}^{\vect{n}-\vect{n}_{\text{min}}}
  \end{align}
  for all $\vect{u} \in (0,\infty)^M$ with $u_1 \le \dots \le u_M$ and all $\vect{n} \in K$.
\end{lemma}

\section{Combinatorial Completion}
\label{s:Completion}  

The proof of Theorem \ref{thm:correlations} follows an inductive argument. Thus, fix $m \ge 2$ and assume $(x(n))$ has $j$-point correlations for all $j < m$. Let $f$ be a $C_c^\infty(\R)$ function, and define
\begin{align*}
  S_N(s,f)=S_N : = \sum_{n \in [N]} \sum_{k \in \Z} f(N(\omega(n) +k  +s)).
\end{align*}
Note that if $f$ was the indicator function of an interval $I$, then $S_N$ would count the number of points in $(x_n)_{n \le N}$ which land in the shifted interval $I/N+s/N$. Now consider the $m^{th}$-moment of $S_N$, then one can show that (see \cite[\S 3]{LutskoTechnau2021})
\begin{align}
  \cM^{(m)}(N) &:= \int_0^1 S_N(s,f)^m \mathrm{d}s\notag\\
              &= \int_0^1 \sum_{\vect{n} \in [N]^m}
                 \sum_{\vect{k} \in \Z^m} 
                 \left( f(N(\omega(n_1)+k_1 + s)) \cdots f(N(\omega(n_m) +k_m + s))\right) \mathrm{d} s \label{m moment}\\
             &=  \frac{1}{N}\sum_{\vect{n} \in  [N]^m} 
                 \sum_{\vect{k}\in \Z^{m-1}}
                 F\left(N(\omega(n_1)-\omega(n_2) +k_1), \dots 
                 N(\omega(n_{m-1})-\omega(n_m) +k_{m-1})\right) \notag,
\end{align} 
where 
\begin{align*}
F(z_1,z_2, \dots , z_{m-1}) := \int_{\R} f(s)f(z_1+z_2+\dots + z_{m-1}+s)f(z_2+\dots + z_{m-1}+s)\cdots f(z_{m-1}+s)\,\rd s. 
\end{align*}
As such we can relate the $m^{th}$ moment of $S_N$ to the $m$-point correlation of $F$. Note that since $f$ has compact support, $F$ has compact support. To recover the $m$-point correlation in full generality, we replace the moment $\int S_N(s,f)^m \mathrm{d}s$ with the mixed moment $\int \prod_{i=1}^m S_N(s,f_i) \mathrm{d} s$ for some collection of $f_i:\R\to \R$. The below proof can be applied in this generality, however for ease of notation we only explain the details in the former case.

In fact, we can use an argument from \cite[\S 8]{Marklof2003} to show that it is sufficient to prove convergence for functions $f$ such that the support of $\wh{f}$ is in $C_c^\infty(\R)$ and $f$ is positive valued. While this implies that the support of $f$ is unbounded, the same argument, together with the decay of Fourier coefficients, applies and we reach the same conclusion about $F$. In the following proof, the support of $\wh{f}$ does not play a crucial role. Increasing the support of $\wh{f}$ increases the range of the $\vect{k}$ variable by a constant multiple. But fortunately in the end we will achieve a very small power saving, so the constant multiple will not ruin the result. To avoid carrying a constant through we assume the support of $\wh{f}$ is contained in $(-1,1)$. Extend all definitions previously made for $f\in C_c^\infty(\R)$ functions to this new class of functions in the obvious way.

\subsection{Combinatorial Target}\label{sec: combinatorial prep}

We will need the following combinatorial definitions to explain how to prove convergence of the $m$-point correlation from \eqref{m moment}. Given a partition $\cP$ of $[ m]$, we say that $j\in [m]$ is \emph{isolated} if $j$ belongs to a partition element of size $1$. A partition is called \emph{non-isolating} if no element is isolated (and otherwise we say it is \emph{isolating}). For our example $\cP = \{\{1,3\}, \{4\}, \{2,5,6\}\}$ we have that $4$ is isolated, and thus $\cP$ is isolating.

Now consider the middle line of \eqref{m moment}, we apply Poisson summation to each of the $k_i$ sums. That is, we insert
\begin{align}
    \sum_{k \in \Z} 
     f(N(\omega(n)+k + s)) = 
     \frac{1}{N}\sum_{k\in \Z} e(k(\omega(n)+s)) \wh{f}(\frac{k}{N}) 
\end{align} 
yielding
\begin{align}\label{M k def}
  \cM^{(m)}(N) = \frac{1}{N^m} \int_0^1 \sum_{\vect{n} \in [N]^m}
  \sum_{\vect{k} \in \Z^m} \wh{f}\Big(\frac{\vect{k}}{N}\Big)
  e( \vect{k}\cdot \omega(\vect{n})     + \vect{k}\cdot \vect{1}s) \mathrm{d} s,
\end{align}
where $\omega(\vect{n}) := (\omega(n_1), \omega(n_2), \dots, \omega(n_m))$ and where $\wh{f}\left(\frac{\vect{k}}{N}\right) = \prod_{i=1}^m \wh{f}\left(\frac{k_i}{N}\right)$.

In \cite[\S 3]{LutskoTechnau2021} we showed that, if
\begin{align*}
  \cE(N):= \frac{1}{N^m} \int_0^1 \sum_{\vect{n} \in [N]^m}\sum_{\substack{\vect{k} \in (\Z^{\ast})^m}} \wh{f}\left(\frac{\vect{k}}{N}\right)e( \vect{k}\cdot \omega(\vect{n})  + \vect{k}\cdot \vect{1}s) \mathrm{d} s,
\end{align*}
then for fixed $m$, and assuming the inductive hypothesis, Theorem \ref{thm:correlations} reduces to the following lemma. 

\begin{lemma}\label{lem:MP = KP non-isolating}
  Let $\mathscr{P}_m$ denote the set of non-isolating partitions of $[m]$.    We have that
\begin{gather}
  \lim_{N \to \infty}\cE(N) =  \sum_{\cP\in \mathscr{P}_m}  \expect{f^{\abs{P_1}}}\cdots \expect{f^{\abs{P_d}}}.\label{m target}
\end{gather}
where the partition $\cP = (P_1, P_2, \dots, P_d)$, and $\abs{P_i}$ is the size of $P_i$.

\end{lemma}

\subsection{Dyadic Decomposition}\label{subsec dyadic}

It is convenient to decompose
the sums over $n$ and $k$ 
within $S_N(s,f)$ into 
(nearly) dyadic ranges in a smooth manner.
Given $N$, we let $Q>1$ be the unique integer with $e^{Q}\leq N < e^{Q+1}$. 
Now, we describe a smooth partition of
unity which approximates the indicator 
function of $[1,N]$. Strictly speaking,
these partitions depend on $Q$, 
however we suppress it from the notation. 
Furthermore, since we want asymptotics of
$\cE(N)$, we need to take a bit of care at
the right end point of $[1,N]$, and so a 
tighter than dyadic decomposition is 
needed. Let us make this precise,
and point out that 
a detailed construction can be found in 
the appendix. For $0\le q < Q$ we let $\mathfrak{N}_q: \R\to [0,1]$ denote a smooth function for which
\begin{align*}
  \operatorname{supp}(\mathfrak{N}_q) \subset [e^{q}/2, 3 e^q),
  \quad \mathrm{for} \quad 
  0\leq q <Q,
\end{align*}
and such that $\mathfrak{N}_{q}(x) + \mathfrak{N}_{q+1}(x) = 1 $ for $x\in [e^q,e^{q+1})$. Now for $q \ge Q$ we let $\mathfrak{N}_q$ form a smooth partition of unity for which 
\begin{equation}\label{eq part unity}
  \sum_{q=0}^{2Q-1} \mathfrak{N}_q (x) =\begin{cases}
  1 & \mbox{if } 1< x < e^{Q}\\
  0 & \mbox{if } x< 1/2 \mbox{ or } x > N + \frac{3N}{\log(N)}
  \end{cases}
\end{equation}
and 
\begin{equation}\label{eq support prop}
    \operatorname{supp}(\mathfrak{N}_q) \subset 
    \left[e^{Q} + (q-Q-1.1)\frac{e^{Q}}{Q}, e^{Q} + (3+q-Q)\frac{e^{Q}}{Q}\right)
    \quad \mathrm{for} \quad Q< q\leq 2Q-1,
\end{equation}
while $\operatorname{supp}(\mathfrak{N}_Q) \subset 
    (0.9\cdot e^{Q-1}, 1.1\cdot e^{Q})$.
Let $\Vert g \Vert_{\infty}$ denote the sup norm of a function $g : \R \to \R$. We impose the following condition on the derivatives:
\begin{align}\label{N deriv}
  \Vert \mathfrak{N}_{q}^{(j)}\Vert_{\infty} \ll \begin{cases}
    e^{-qj} & \mbox{ for } q < Q\\ 
    (e^{Q}/Q)^{-j} & \mbox{ for } Q< q,
  \end{cases}
\end{align}
for $j \ge 1$. For technical reasons, assume $\mathfrak{N}_q^{(1)}$ changes sign only once.

Notice that \eqref{eq part unity} implies 
$$
\sum_{n \in \Z}\sum_{q=0}^{2Q-2} \mathfrak{N}_q (n) \sum_{k \in \Z} f(N(\omega(n) +k  +s))
\leq 
S_N(s,f) \leq 
\sum_{n\in \Z}\sum_{q=0}^{2Q-1} \mathfrak{N}_q (n) 
\sum_{k \in \Z} f(N(\omega(n) +k  +s)).
$$
Ignoring the lower bound, which can be treated similarly, applying Poisson summation we then have
$$
S_N(s,f)\leq 
\frac{1}{N}
\sum_{q=0}^{2Q-1} \mathfrak{N}_q(n)
\sum_{k \in \Z} \widehat{f}(k/N)
e(k(\omega(n) +s)).
$$
Next, by positivity, 
we have that
\begin{align}\label{Mm ineq}
\cM^{(m)}(N)  \le \int_0^1  
  \bigg( \frac{1}{N} 
  \sum_{q=0}^{2Q-1} \sum_{n \in \Z} 
  \mathfrak{N}_{q}(n)   
  \sum_{k \in \Z} 
  \wh{f}\left(\frac{ k}{N}\right)   
  e(   k \omega(n) +  ks)
  \bigg)^m \mathrm{d} s. 
\end{align}
All frequencies $\vect{k}$
for which $k_j=0$ for at least
some index $1\leq j\leq n$
contribute 
to $\cM^{(m)}(N)$ exactly
$$
  \sum_{i=1}^n \begin{pmatrix}
      n\\ i
  \end{pmatrix}\wh{f}\left(0\right)^{i}\int_0^1  
 \bigg( \frac{1}{N} 
 \sum_{n \in [N]} 
  \sum_{k \neq 0}  
  \wh{f}\left(\frac{ k}{N}\right)   
  e(   k \omega(n) +  ks)
  \bigg)^{m-i} \mathrm{d} s.
$$
Subtracting exactly the above term from both sides of \eqref{Mm ineq}, while
using our inductive assumption 
that $M^{m-i}(N)$ converge 
(for $1\leq i\leq m-2)$,
then yields



\begin{equation}\label{eq: upper bound moment}
  \cE(N) \le \int_0^1  
  \bigg( \frac{1}{N} 
  \sum_{q=0}^{2Q-1} \sum_{n \in \Z} 
  \mathfrak{N}_{q}(n)   
  \sum_{k\neq 0} 
  \wh{f}\left(\frac{ k}{N}\right)   
  e(   k \omega(n) +  ks)
  \bigg)^m \mathrm{d} s +o(1).
\end{equation}
 The same argument can be used to yield, 
 $$
 \cE(N)+o(1) \ge \int_0^1  
  \bigg( \frac{1}{N} 
  \sum_{q=0}^{2Q-2} \sum_{n \in \Z} 
  \mathfrak{N}_{q}(n)   
  \sum_{k\neq 0} 
  \wh{f}\left(\frac{ k}{N}\right)   
  e(   k \omega(n) +  ks)
  \bigg)^m \mathrm{d} s .
 $$

We similarly decompose the $k$ sums, although thanks to the compact support of $\wh{f}$ we do not need to worry about $k\ge N$. Let $\mathfrak{K}_u:\R \to [0,1]$ be a smooth function such that, for $U :=\lceil \log N \rceil$
\begin{align*}
  \sum_{u=-U}^{U} \mathfrak{K}_u(k) =\begin{cases}
  1 & \mbox{ if } \vert k\vert  \in [1, N)\\
  0 & \mbox{ if } \vert k\vert < 1/2 \mbox{ or } \vert k\vert > 2N,
  \end{cases}
\end{align*}
and the symmetry $\mathfrak{K}_{-u}(k) = \mathfrak{K}_{u}(-k)$ holds true for 
all $u,k> 0$. Similarly
\begin{align*}
  &\supp(\mathfrak{K}_u) \subset[e^{u }/3, 3e^{u})
  \qquad\qquad \mbox{ if } u \ge 0\,\,,\mbox{ and } \\
  &
  \Vert \mathfrak{K}_{u}^{(j)} \Vert_{\infty} \ll e^{-\abs{u}j},\qquad\qquad
  \mbox{for all } j\ge 1 .
\end{align*}
As for $\mathfrak{N}_q$, we also assume $\mathfrak{K}_u^{(1)}$ changes sign exactly once.

Therefore a central role is played by the smoothed exponential sums
\begin{equation}\label{def: E_qu}
  \cE_{q,u}(s):=
  \frac{1}{N}\sum_{k\in \Z} \mathfrak{K}_{u}(k) 
  \wh{f}\Big(\frac{k}{N}\Big)e( ks)
  \sum_{n\in \Z} \mathfrak{N}_{q}(n)e(   k\omega(n) ).
\end{equation}
Notice that \eqref{eq: upper bound moment} and the compact support of $\wh{f}$ imply
\begin{align*}
\cE(N) \ll \bigg\Vert \sum_{u=-U}^{U}\sum_{q=0}^{2Q-1} \cE_{q,u} 
\bigg\Vert_{L^m(\R)}^m.
\end{align*}
Now write
\begin{align*}
  \cF(N) :=   \frac{1}{N^m} 
  \sum_{\vect{q}\in [0,2Q-1]^m}
  \sum_{\vect{u} = [-U,U]^m} 
  \sum_{\vect{k},\vect{n} \in \Z^m} 
  \mathfrak{K}_{\vect{u}}(\vect{k})
  \mathfrak{N}_{\vect{q}}(\vect{n})   
  \int_0^1\wh{f}\Big(\frac{\vect{k}}{N}\Big)   
  e( \vect{k}\cdot \omega(\vect{n})  + \vect{k}\cdot \vect{1} s)\, \mathrm{d} s,
\end{align*}
where $\mathfrak{N}_{\vect{q}}(\vect{n}) := \mathfrak{N}_{q_1}(n_1)\mathfrak{N}_{q_2}(n_2)\cdots\mathfrak{N}_{q_m}(n_m)$ and $\mathfrak{K}_{\vect{u}}(\vect{k}) := \mathfrak{K}_{u_1}(k_1)\mathfrak{K}_{u_2}(k_2)\cdots\mathfrak{K}_{u_m}(k_m)$. Our goal will be to establish that $\cF(N)$ is equal to the right hand side of \eqref{m target} up to a $o(1)$ term. Then, since we can establish the same asymptotic for the lower bound, we may conclude the asymptotic for $\cE(N)$. Since the details are identical, we will only focus on $\cF(N)$.

Fixing $\vect{q}$ and $\vect{u}$, we let 
\begin{align*}
  \cF_{\vect{q},\vect{u}}(N) 
  & =
   \frac{1}{N^m}\int_0^1  
   \sum_{\substack{\vect{n},\vect{k} \in \Z^m}} 
   \mathfrak{N}_{\vect{q}}(\vect{n})
   \mathfrak{K}_{\vect{u}}(\vect{k})   
   \wh{f}\Big(\frac{\vect{k}}{N}\Big)   
   e(\vect{k}\cdot \omega(\vect{n}) +  \vect{k}\cdot \vect{1}s) \mathrm{d} s.
\end{align*}

\begin{remark}
  In the proceeding sections, we will fix $\vect{q}$ and $\vect{u}$. Because of the way we have defined $\mathfrak{N}_q$, this implies two cases: $q<Q$ and $q\ge Q$. The only real difference in these two cases are the bounds in \eqref{N deriv}, which differ by a factor of $Q = \log(N)$. To keep the notation simple, we will assume we have $q <Q$ and work with the first bound. In practice the logarithmic correction does not affect any of the results or proofs.
\end{remark}

\section{Applying the $B$-process}
\label{s:Applying B}

\subsection{Degenerate Regimes}

Fix  $\delta=\frac{1}{m+1}$. We say $(q,u)\in [2 Q] \times [-U,U]$ is \emph{degenerate} if either one of the following holds
\begin{align*}
 \abs{u} < q^\frac{A-1}{2},
  \,\, \mathrm{or}\,\,
  q \leq \delta Q.
\end{align*}
Otherwise $(q,u)$ is called \emph{non-degenerate}. Let $\mathscr{G}(N)$ denote the set of all non-degenerate pairs $(q,u)$. In this section it is enough to suppose that $u>0$ (and therefore $k>0$). Next, we show that degenerate $(q,u)$ contribute a negligible amount to $\cF(N)$.

First, assume $q \le \delta Q$. Expanding the $m^{th}$-power, evaluating the $s$-integral and trivial estimation yield
\begin{align*}
  \Vert \cE_{q,u} \Vert_{L^m}^m
    &=
    \int_0^1\left(\frac{1}{N} \sum_{k \in \Z} \mathfrak{K}_u(k)\wh{f}\left(\frac{k}{N}\right) e(ks) \sum_{n \in \Z} \mathfrak{N}_{q}(n)e(k\omega(n))\right)^m \mathrm{d}s\\
    &\ll
    \frac{1}{N^m}\sum_{\substack{k_i \asymp e^u,\\ i =1,\dots m}}  \max_x\left(\wh{f}\left(\frac{x}{N}\right)\right)^m\sum_{\substack{n_i \sim e^q,\\ i =1,\dots m}} \bigg|\int_0^1 e((k_1+ \dots + k_m)s) \mathrm{d}s\bigg|
    \\
  &\ll \frac{1}{N^m} \#\{k_1,\dots,k_m\asymp e^{u}: k_1 +\dots +k_m = 0 \}N^{m\delta}
  \ll N^{m\delta-1}.
\end{align*}
If $u<q^{(A-1)/2}$ and $q>\delta Q$,
then we can apply the Euler summation formula, followed by van der Corput's lemma with $j=1$, to conclude that 
\begin{align*}
  \sum_{n\in \Z} 
  \mathfrak{N}_{q}(n)      
  e( k \omega(n)) \ll \frac{e^q}{k q^{A-1}},
\end{align*}
where the numerator is the size of the support of $\mathfrak{N}_q$ and the denominator is the maximum value of $k \omega'(x)$ for $x$ in that support. Hence
\begin{align*}
  \Vert \cE_{q,u} \Vert_\infty 
  \ll \frac{1}{N} 
     \sum_{k\asymp e^{u}}
     \frac{e^q}{k q^{A-1}}
  \ll \frac{1}{N} 
     \frac{e^q}{ q^{A-1}}.
\end{align*}
Note 
\[
\sum_{q\leq Q}\frac{e^{q}}{q^{A-1}}\ll\int_{1}^{Q}\frac{e^{q}}{q^{A-1}}\,\mathrm{d}q=\int_{1}^{Q/2}\frac{e^{q}}{q^{A-1}}\,\mathrm{d}q+\int_{Q/2}^{Q}\frac{e^{q}}{q^{A-1}}\,\mathrm{d}q\ll e^{Q/2}+\frac{1}{Q^{A-1}}\int_{Q/2}^{Q}e^{Q}\,\mathrm{d}q\ll\frac{e^{Q}}{Q^{A-1}}.
\]
Thus,
\[
\Big\Vert \sum_{\delta Q\leq q\leq Q} \sum_{u\leq q^{(A-1)/2}} 
\mathcal{E}_{q,u} \Big\Vert_\infty \ll
\frac{1}{N}
\sum_{q\leq Q} \sum_{u\leq q^{(A-1)/2}}\sum_{k\asymp e^{u}}\frac{1}{k}
\frac{e^{q}}{q^{A-1}}
\ll
\frac{1}{N} \frac{e^{Q}}{Q^{A-1}}
\sum_{u\leq Q^{(A-1)/2}}1\leq\frac{1}{Q^{\frac{A-1}{2}}}.
\]
Taking the $L^m$-norm then yields:
\begin{align*}
  \bigg \Vert  \sum_{(q,u)\in [2Q]\times[-U,U] \setminus \mathscr{G}(N)}\cE_{q,u} 
  \bigg\Vert_{L^m}^m \ll_{\delta} \log(N)^{-\rho},
\end{align*}
for some $\rho>0$. Hence the triangle inequality implies
\begin{equation}\label{eq: cal F reduced}
  \mathcal{F}(N) = 
  \bigg\Vert \sum_{(u,q) \in \mathscr{G}(N)}\cE_{q,u} \bigg\Vert_{L^m}^m + O(N^{-\rho}).
\end{equation}

Next, to dismiss the degenerate regimes, let $w,W$ denote strictly positive numbers satisfying
$w<W$. Consider
\[
g_{w,W}(x):=\min\left(\frac{1}{\Vert xw\Vert},\frac{1}{W}\right),
\]
here $\| \cdot\|$ denotes the distance to the nearest integer. We shall need (as in \cite[Proof of Lemma 4.1]{LutskoTechnau2021}):
\begin{lemma}
\label{lem: square_root_cancellation on average over s}If $W<1/10$,
then 
\[
\sum_{e^{u}\leq\left|k\right|<e^{u+1}}g_{w,W}(k)\ll\left(e^{u}+\frac{1}{w}\right)\log\left(1/W\right)
\]
where the implied constant is absolute.
\end{lemma}

\begin{proof}
  The proof is elementary, hence we only sketch the main idea. If $e^{u}w <1$ then we achieve the bound $\frac{1}{w} \log(1/W)$, and otherwise we get the bound $e^{u}\log(1/W)$. Focusing on the latter, first make a case distinction between those $x$ which contribute $\frac{1}{\|xw\|}$ and those that contribute $\frac{1}{W}$. Then count how many contribute the latter. For the former, since the spacing between consecutive points is small, we can convert the sum into $e^uw$ many integrals of the form $\frac{1}{w}\int_{W/w}^{1/w} \frac{1}{x} \mathrm{d}x$.

\end{proof}

With the previous lemma at hand, we can show that an additional degenerate regime is negligible. Specifically, when we apply the $B$-process, the first step is to apply Poisson summation. Depending on the new summation variable there may, or may not, be a stationary point. The following lemma allows us to dismiss the contribution when there is no stationary point. Fix $k\asymp e^{u}$ and let $[a,b]:=\mathrm{supp}(\mathfrak{N}_{q})$. Consider 
\[
  \mathrm{Err}(k):=\sum_{\underset{m_{q}(r)>0}{r\in\mathbb{Z}}}\int_{\mathbb{R}}\,e(\Phi_r(x))\,\mathfrak{N}_{q}(x)\,\mathrm{d}x
\]
where 
\[
  \Phi_r(x):=k \omega(x)-rx,\qquad m_{q}(r):=\min_{x\in[a,b]}\vert\Phi_{r}(x)\vert.
\]
Our next aim is to show that the smooth exponential sum
\[
\mathrm{Err}_{u}(s):=
\sum_{k\in\mathbb{Z}}e(ks)\mathrm{Err}(k)\mathfrak{K}_{u}(k)
\widehat{f}\Big( \frac{k}{N}\Big)
\]
is small on average over $s$:

\begin{lemma}\label{lem: error small in L_m average}
  Fix any constant $C>0$, then the bound
  \begin{align}
    I_{\vect{u}}:=\int_0^1\prod_{i=1}^m \Err_{u_i}(s)\mathrm{d}s\ll Q^{-C}N^{m},
  \end{align}
  holds uniformly in $Q^{\frac{A-1}{2}}\leq \vect{u}\ll Q$.
\end{lemma}

\begin{proof}

  Let $\mathcal{L}_{\vect{u}}$ denote the truncated sub-lattice of $\mathbb{Z}^{m}$ defined by gathering all $\vect{k}\in\mathbb{Z}^{m}$ so that $k_{1}+\ldots+k_{m}=0$ and $\vert k_{i}\vert\asymp e^{u_i}$ for all $i\leq m$. The quantity $\mathcal{L}_{\vect{u}}$ arises from
  \begin{equation}
    I_{\vect{u}}=
    \sum_{\underset{i\leq m}{\left|k_{i}\right|\asymp e^{u}}}
    \Biggl(\Biggl(\prod_{i\leq m}
    \mathrm{Err}(k_{i})\mathfrak{K}_{u}(k_{i})\widehat{f}\Big( \frac{k_i}{N}\Big)
    \Biggr)\int_{0}^{1}e((k_{1}+\ldots+k_{m})s)\,\mathrm{d}s\Biggr)
    \ll \sum_{\vect{k}\in\mathcal{L}_{\vect{u}}}
    \biggl(\prod_{i\leq m}\mathrm{Err}(k_{i})\biggr).\label{eq: bound on L_m norm of Err}
  \end{equation}
  Partial integration, and the dyadic decomposition allow one to show that the contribution of $\vert r\vert \geq Q^{O(1)}$ to $\Err(k_i)$ can be bounded by  $O(Q^{-C})$ for any $C>0$. Hence, from van der Corput's lemma (Lemma \ref{lem: van der Corput's lemma}) with $j=2$  and the assumption $m_{q}(r)>0$, we infer
  \[
  \mathrm{Err}(k)\ll Q^{O(1)}\min\left(\frac{1}{\Vert k\omega'(a)-r\Vert},\frac{1}{(k\omega''(a))^{1/2}}\right)= Q^{O(1)} \min\left(\frac{1}{\Vert k\omega'(a)\Vert},\frac{1}{(k\omega''(a))^{1/2}}\right)
  \]
  where the implied constant is absolute. Notice that $\omega'(a)\asymp q^{A-1}e^{-q}=:w$, and 
  \[
  k\omega''(a)\asymp(e^{u-2q}q^{A-1})^{1/2}=:W.
  \]
  Thus $\mathrm{Err}(k)\ll g_{w,W}(k) Q^{O(1)}.$ Using $\mathrm{Err}(k_{i})\ll g_{w,W}(k_{i}) Q^{O(1)}$ for $i<m$ and $\mathrm{Err}(k_{m})\ll  Q^{O(1)}/W$ in (\ref{eq: bound on L_m norm of Err}) produces the estimate
  \begin{equation}
    I_{\vect{u}}\ll\frac{Q^{O(1)}}{W}
    \sum_{\underset{i<m}{\vert k_{i}\vert\asymp e^{u}}}
    \biggl(\prod_{i<m}g_{w,W}(k_{i})\biggr)
    =\frac{Q^{O(1)}}{W}\biggl(\sum_{\vert k \vert
      \asymp e^{u}}g_{w,W}(k)\biggr)^{m-1}.\label{eq: interm L_m normbnd}
  \end{equation}
  Suppose $W\geq N^{-\varepsilon}$, then $g_{w,W}(k)\leq N^{\varepsilon}$
  and we obtain that 
  \[
  I_{\vect{u}}\ll Q^{O(1)} N^{\varepsilon m}e^{u_1+\dots+u_{m-1}}\ll N^{m-1+\varepsilon m}\ll Q^{-C}N^{m}.
  \]
  Now suppose $W<N^{-\varepsilon}\leq1/10$. Then Lemma \ref{lem: square_root_cancellation on average over s}
  is applicable and yields
  \[
  \sum_{\vert k\vert\asymp e^{u}}g_{w,W}(k)
  \ll\left(e^{u}+1/w\right)\log\left(1/W\right)
  \ll\left(e^{u}+e^{q}\right)\log\left(1/W\right)\ll NQ.
  \]
  Plugging this into \eqref{eq: interm L_m normbnd} and using 
  $1/W\ll e^{q-u/2}q^{(1-A)/2}\ll Ne^{-\frac{u}{2}}$
  shows that
  \[
  I_{\vect{u}}\ll Q^{O(1)} 
  \frac{(NQ)^{m-1}}{W}\ll Q^{O(1)} (NQ)^{m}e^{-\frac{u}{2}}.
  \]
  Because $u\geq Q^{\frac{A-1}{2}}$, we certainly have $e^{-\frac{u}{2}}\ll Q^{-C}$ for any $C>0$ and thus the proof is complete.

\end{proof}

\subsection{First application of the $B$-Process}
\label{ss:B proc trip}

First, following the lead set out in \cite{LutskoSourmelidisTechnau2021} we apply the $B$-process in the $n$-variable. Assume without loss of generality that $k >0$ (if $k <0$ we take complex conjugates and the w.l.o.g. assumption that $f$ is even). 

Given $r \in \Z$, let $x_{k,r}$ denote the stationary point of the function $k\omega(x) - rx$, thus:
\begin{align*}
  x_{k,r} : = \wt{\omega}\left(\frac{r}{k}\right),
\end{align*}
where $\wt{\omega}(x):= (\omega^{\prime})^{-1}(x)$, the inverse of the derivative of $\omega$. This is well defined as long as $x > e^{A-1}$ (the inflection point of $\omega$) which is satisfied in the non-degenerate regime. Then, after applying the $B$-process, the phase will be transformed to 
\begin{align*}
  \phi(k,r): = k\omega\left(x_{k,r}\right) - r x_{k,r}.
\end{align*}
With that, the next lemma states that $\cE_{q,u}$ is well-approximated by
\begin{align}\label{EB def}
  \cE_{q,u}^{(B)}(s):= 
  \frac{e(-1/8)}{N}\sum_{k\geq 0} \mathfrak{K}_{ u}(k)\wh{f}\Big(\frac{ k}{N}\Big) e(ks)
  \sum_{r\geq 0} \frac{\mathfrak{N}_{q}(x_{k,r})}{\sqrt{k\omega^{\prime\prime}(x_{k,r})}}e(\phi(k,r)).
\end{align}

\begin{proposition}\label{prop:E B triple}
   If $u\geq Q^{(A-1)/2}$, then 
   \begin{align}
      \Vert \cE_{q,u} -  \cE^{(B)}_{q,u}\Vert_{L^m}^m \ll Q^{-100m},
   \end{align}
   uniformly for all non-degenerate $(u,q)\in \mathscr{G}(N)$.
\end{proposition}

\begin{proof}
 Let $[a,b]:= \supp(\mathfrak{N}_q)$, let $\Phi_r(x):= k\omega(x) - rx$, and let $m(r):= \min \{\abs{\Phi_r^\prime(x)} \, : \, x \in [a,b]\}$. As usual when applying the $B$-process we first apply Poisson summation and integration by parts:
  \begin{align*}
    \sum_{n \in \Z} \mathfrak{N}_q(n) e(k\omega(n)) = \sum_{r \in \Z} \int_{-\infty}^\infty \mathfrak{N}_q(x) e(\Phi_r(x)) \mathrm{d} x = M(k) + \Err(k),
  \end{align*}
  where $M(k)$ gathers the contributions when $r\in\Z$ with $m(r)=0$ (i.e with a stationary point) and $\Err(k)$ gathers the contribution of $0 < m(r)$. 

  In the notation of Lemma \ref{lem: stationary phase}, let $w(x):= \mathfrak{N}_q(x)$, $\Lambda_{\psi} := \omega(e^q)e^{u} = q^{A}e^{u}$, and $\Omega_\psi = \Omega_{w} := e^q$. Since $(u,q)$ is non-degenerate we have that $\Lambda_\psi/\Omega_\psi \gg q$, and hence
  \begin{align}\label{eq: stationary phase main term}
    M(k) = e(-1/8)
  \sum_{r\geq 0} \frac{\mathfrak{N}_{q}(x_{k,r})}{\sqrt{k\omega^{\prime\prime}(x_{k,r})}}e(\phi(k,r)) +  O\left(\left(q \Lambda_\psi^{1/2+O(\varepsilon)}\right)^{-1}\right).
  \end{align}
  Summing \eqref{eq: stationary phase main term} against $ N^{-1} \mathfrak{K}_{u}(k)  \wh{f}( k/N) e(ks)$ for $k\geq 0$ gives rise to $\mathcal{E}_{q,u}^{\mathrm{(B)}}$. The term coming from
  \begin{align*}
    \Err(k) N^{-1} \mathfrak{K}_{u}(k)  \wh{f}( k/N) e(ks) = 
    \frac{1}{N}\mathrm{Err}_{u}(s)
  \end{align*}
  can be bounded sufficiently by Lemma \ref{lem: error small in L_m average} and the triangle inequality.

\end{proof}

Since $x_{k,r}$ is roughly of size $e^{q}$, if we stop here, and apply the triangle inequality to \eqref{EB def} we would get
\begin{align}\label{EB def}
  \abs{\cE_{q,u}^{(B)}(s)}\ll 
  \frac{1}{N}\sum_{k\geq 0} \mathfrak{K}_{ u}(k)  e^q \frac{1}{\sqrt{k}}  \frac{k}{e^q} \ll 
  \frac{1}{N} e^{3u/2} \ll  N^{1/2}.
\end{align}
Hence, we still need to find a saving of $O(N^{1/2})$. To achieve most of this, we now apply the $B$-process in the $k$ variable. This will require the following a priori bounds.

  \subsection{Amplitude Bounds}
  \label{ss:Amplitude}

  Before proceeding with the second application of the $B$-process, we require bounds on the amplitude function
  \begin{align*}
    \Psi_{q,u}(k,r,s) = \Psi_{q,u} : = \frac{\mathfrak{N}_q(x_{k,r}) \mathfrak{K}_u(k)}{\sqrt{k \omega^{\prime\prime}(x_{k,r})}} \wh{f}\left(\frac{k}{N}\right),
  \end{align*}
  and its derivatives; for which we have the following lemma
  \begin{lemma} \label{lem:amp bounds}
    For any pair $q,u $ as above, and any $j \ge 1$, we have the following bounds
    \begin{align}\label{amp bounds}
      \|\partial_k^j \Psi_{q,u}(k,r,\cdot)\|_\infty 
      \ll e^{-uj} Q^{O(1)}\Vert \Psi_{q,u} \Vert_{\infty} 
    \end{align}
    where the implicit constant in the exponent depends on $j$, but not $q,u$. Moreover
    \begin{align*}
      \|\Psi_{q,u}\|_\infty \ll e^{q-u/2} q^{-\frac{1}{2}(A-1)}.
    \end{align*}
  \end{lemma}

  \begin{proof}
    First note that since $\Psi_{q,u}$ is a product of functions of $k$, if we can establish \eqref{amp bounds} for each of these functions, then the overall bound will hold for $\Psi_{q,u}(k,r,s)$ by the product rule. Moreover the bound is obvious for $\mathfrak{K}_u(k)$, $\wh{f}(k/N)$, and $k^{-1/2}$.

    Thus consider first $\partial_k \mathfrak{N}_q(x_{k,r}) = \mathfrak{N}_q^\prime(x_{k,r})  \partial_k(x_{k,r}) $. By assumption since $x_{k,r} \asymp e^{q}$, we have that $\mathfrak{N}_q^\prime(x_{k,r}) \ll e^{-q}$. Again, by repeated application of the product rule, it suffices to show that $\partial_k^j x_{k,r} \ll e^{q-uj}Q^{O(1)}$. To that end, begin with the following equation
    \begin{align*}
      1 = \partial_x(x) =\partial_x( \wt{\omega}(\omega^{\prime}(x)))= \wt{\omega}^\prime (\omega^\prime(x)) \omega^{\prime\prime}(x).
    \end{align*}
    Hence $\wt{\omega}^\prime (\omega^\prime(x)) = \frac{1}{\omega^{\prime\prime}(x)}$ which we can write as
    \begin{align*}
      \wt{\omega}^\prime (\omega^\prime(x)) = x^2 f_1(\log(x))
    \end{align*}
    where $f_1$ is a rational function. Now we take $j-1$ derivatives of each side. Inductively, one sees that there exist rational functions $f_j$ such that
    \begin{align*}
      \wt{\omega}^{(j)} (\omega^\prime(x)) = x^{j+1} f_j(\log(x)).
    \end{align*}
    Setting $x= x_{k,r}=\wt\omega(r/k)$ then gives
    \begin{align}\label{omega tilde bound}
      \wt{\omega}^{(j)} (r/k) = x_{k,r}^{j+1} f_j(\log(x_{k,r})).
    \end{align}

    With \eqref{omega tilde bound}, we can use repeated application of the product rule to bound
    \begin{align*}
      \partial_k^{j} x_{k,r} &= \partial_k^{j} \wt{\omega}(r/k)\\
      & = -\partial_k^{j-1} \wt{\omega}^\prime (r/k)\left(\frac{r}{k^2}\right)\\
      & \ll \wt{\omega}^{(j)} (r/k)\left(\frac{r}{k^2}\right)^{j} + \wt{\omega}^\prime (r/k)\left(\frac{r}{k^{1+j}}\right)\\
      & \ll x_{k,r}^{j+1} f_j(\log(x_{k,r}))\left(\frac{r}{k^2}\right)^{j} + x_{k,r}^2 f_1(\log(x_{k,r}))\left(\frac{r}{k^{1+j}}\right).
    \end{align*}
    Now recall that $k \asymp e^{u}$, $x_{k,r} \asymp e^{q}$, and $r \asymp e^{u-q}q^{A-1}$, thus
    \begin{align*}
      \partial_k^{j} x_{k,r} & \ll \left(e^{q(j+1)} \left(\frac{e^{u-q}}{e^{2u}}\right)^{j} + e^{2q}\left(\frac{e^{u-q}}{e^{(1+j)u}}\right)\right)Q^{O(1)}\\
       & \ll  e^{q-ju}Q^{O(1)}.
    \end{align*}
    Hence $\partial_k^{(j)} \mathfrak{N}_q(x_{k,r}) \ll e^{-ju} Q^{O(1)}$.

    The same argument suffices to prove that $\partial_k^{j}\frac{1}{\sqrt{\omega^{\prime\prime}(x_{k,r})}} \ll e^{q-ju} Q^{O(1)}$.

  \end{proof}

  \subsection{Second Application of the $B$-Process}
  \label{s:Second Application}
  Now, we apply the $B$-process in the $k$-variable. At the present stage, the phase function is $\phi(k,r)+ks$. Thus, for $h \in \Z$ let $\mu = \mu_{h,r,s}$ be the unique stationary point of $k \mapsto\phi(k,r) - (h-s)k$. Namely:
  \begin{align*}
    (\partial_\mu\phi)(\mu,r) = h-s.
  \end{align*}
  After the second application of the $B$-process, the phase will be transformed to
  \begin{align*}
    \Phi(h,r,s) = \phi(\mu,r) - (h-s)\mu.
  \end{align*}
  With that, let (again for $u >0)$
  \begin{align} \label{EBB def}
    \cE^{(BB)}_{q,u}(s):= \frac{1}{N} \sum_{r\ge 0}\sum_{h \ge 0}
    \wh{f} \left(\frac{\mu}{N}\right)\mathfrak{K}_u(\mu)\mathfrak{N}_{q}(x_{\mu,r}) \frac{1}{\sqrt{\abs{\mu \omega^{\prime\prime}(x_{\mu,r}) \cdot (\partial_{\mu\mu}\phi)(\mu,r) }}} e(\Phi(h,r,s)).
  \end{align}

  We can now apply the $B$-Process for a second time and conclude
  \begin{proposition} \label{B proc twice}
    We have
    \begin{align}
      \left\Vert \mathcal{E}_{q,u}^{\mathrm{(BB)}}-\mathcal{E}_{q,u}^{\mathrm{(B)}}\right\Vert _{L^{m}([0,1])}=O(N^{-\frac{1}{2m} +\varepsilon}),
    \end{align}
    uniformly for any non-degenerate $(q,u)\in \mathscr{G}(N)$.
  \end{proposition} 
  Before we can prove the above proposition, we need some preparations. Note the following: we have 
  \[
  k\omega'(n)=Ak \frac{(\log n)^{A-1}}{n}\leq 10 A e^{u-q}q^{A-1}.
  \]
  If $u-q+(A-1)\log q<-10$ then $10 A e^{u-q}q^{A-1}=10 Ae^{-10 A}\leq0.6$. Hence, there is no stationary point in the first application of the $B$-process. Thus the contribution from this regime is disposed of by the first $B$-process. Therefore, from now on we assume that 
  \begin{equation}
    u\geq q-(A-1)\log q-10A,\quad\mathrm{in\,particular}\quad e^{u}\gg e^{q}q^{1-A}.\label{eq: lower bound on u in terms of q}
  \end{equation}

\subsection*{Non-essential regimes}

In this section we estimate the contribution from regimes where $u$ is smaller by a power of a logarithm than the top scale $Q$. We shall see that this regime can be disposed off. More precisely, let 
\[
  \mathcal{T}(N):=\{(q,u)\in\mathscr{G}(N):\,u\leq\log N- 10 A\log\log N\}.
\]
We shall see that contribution $\mathcal{T}(N)$ is negligible by showing that the function
\begin{equation}
  T_{N}(s):=\sum_{(q,u)\in\mathcal{T}(N)}\mathcal{E}_{q,u}^{(\mathrm{B})}(s)\label{def: tiny contribution}
\end{equation}
has a small $\Vert\cdot\Vert_{\infty}$-norm (in the $s\in [0,1]$ variable). To prove this, we need to ensure that in
\[
  \mathcal{E}_{q,u}^{(\mathrm{B})}(s)=\frac{e(-1/8)}{N}\sum_{r\geq0}\sum_{k\geq0}\Psi_{q,u}(k,r,s)e(\phi(k,r)-ks)
\]
the amplitude function
\[
  \Psi_{q,u}(k,r,s):=\frac{\mathfrak{N}_{q}(x_{k,r})
  \mathfrak{K}_{u}(k)}{\sqrt{k\omega''(x_{k,r})}}\widehat{f}\left(\frac{k}{N}\right)
\]
has a suitably good decay in $k$. 

\begin{lemma} \label{lem: L1 bound}
  If (\ref{eq: lower bound on u in terms of q}) holds, then 
  \[
  \left\Vert k\mapsto\partial_{k}
  \Psi_{q,u}(k,r,s)\right\Vert_{L^{1}(\mathbb{R})}
  \ll e^{u/2} q^{-\frac{1}{2}(A-1)},
  \]
  uniformly for $r$ and $s$ in the prescribed ranges.
\end{lemma}

\begin{proof}
  First use the triangle inequality to bound
  \begin{align*}
    \left\Vert k\mapsto\partial_{k}
    \Psi_{q,u}(k,r,s)\right\Vert_{L^{1}(\mathbb{R})}
    \ll \left\Vert \partial_k\left\{ \frac{\mathfrak{N}_{q}(x_{k,r})
  \mathfrak{K}_{u}(k)}{\sqrt{k\omega''(x_{k,r})}}\right\}\widehat{f}\left(\frac{k}{N}\right)  \right\Vert_{L^{1}(\mathbb{R})} +  \left\Vert  \frac{\mathfrak{N}_{q}(x_{k,r})
  \mathfrak{K}_{u}(k)}{\sqrt{k\omega''(x_{k,r})}}\partial_k\widehat{f}\left(\frac{k}{N}\right)  \right\Vert_{L^{1}(\mathbb{R})}.
  \end{align*}
  Since $\wh{f}$ has bounded derivative, the term on the right can be bounded by $1/N$ times the sup norm times $e^{u}$.
  Since $\widehat{f}\left(\frac{k}{N}\right)$ is bounded, and $\frac{\mathfrak{N}_{q}(x_{k,r})
  \mathfrak{K}_{u}(k)}{\sqrt{k\omega''(x_{k,r})}}$ changes sign finitely many times, we can apply the fundamental theorem of calculus and bound the whole by 
  \begin{align*}
    \left\Vert k\mapsto\partial_{k}
    \Psi_{q,u}(k,r,s)\right\Vert_{L^{1}(\mathbb{R})}
    \ll \left\Vert k \mapsto   \frac{1}{\sqrt{k\omega''(x_{k,r})}}  \right\Vert_{L^{\infty}(\mathbb{R})}.
  \end{align*}

\end{proof}

Now we are in the position to prove that the contribution from \eqref{def: tiny contribution} is negligible thanks to a second derivative test. This is one of the places where, in contrast to the monomial case, we only win by a logarithmic factor. Moreover, this logarithmic saving goes to $0$ as $A$ approaches $1$. 
\begin{lemma}
  \label{lem: second derivative bound}The oscillatory integral
  \begin{align} \label{I def}
    I_{q,u}(h,r):=\int_{-\infty}^{\infty}\Psi_{q,u}(t,r,s)e(\phi(t,r)-t(h-s))\,\mathrm{d}t,
  \end{align}
  satisfies the bound
  \begin{equation}
    I_{q,u}(h,r)\ll e^{q} q^{1-A}\label{eq: bound on oscillatory integral second B}
  \end{equation}
  uniformly in $h,$ and $r$ in ranges prescribed by $\Psi$.
\end{lemma}

\begin{proof}
We aim to apply van der Corput's lemma (Lemma \ref{lem: van der Corput's lemma}) for a second derivative bound. For that, first note that $\partial_{t}\phi(t,r)=\omega(x_{t,r})+t\partial_{t}(\omega(x_{t,r}))-r\partial_{t}(x_{t,r})$. Now, since 
\begin{equation}
\partial_{t}(\omega(x_{t,r}))=\omega'(x_{t,r})\partial_{t}(x_{t,r})=\frac{r}{t}\partial_{t}(x_{t,r}),\label{eq: derivative of omega xkr}
\end{equation}
it follows that
\begin{equation}
\partial_{t}\phi(t,r)=\omega(x_{t,r}).\label{eq: second B 1 deriv comp}
\end{equation}

Now we bound the second derivative of $\phi(t,r)-t(s+h)$. By \eqref{eq: derivative of omega xkr} and \eqref{eq: second B 1 deriv comp}, we have
\begin{align*}
\partial_{t}^{2}[\phi(t,r)] & =\partial_{t}[\omega(x_{t,r})]=\frac{r}{t}\partial_{t}[x_{t,r}].
\end{align*}
Thus
\begin{align*}
  \partial_{t}^{2}[\phi(t,r)]=-\frac{1}{\omega''(x_{t,r})}\frac{r^{2}}{t^{3}}.
\end{align*}
Taking $x_{t,r}\asymp e^{q}$ into account gives
\begin{align}
  \partial_{t}^{2}[\phi(t,r)]\asymp\frac{1}{e^{-2q}q^{A-1}}\frac{(e^{u-q}q^{A-1})^{2}}{e^{3u}}=e^{-u}q^{A-1}.\label{eq: size of second deriv in k}
\end{align}
The upshot, by van der Corput's lemma (Lemma \ref{lem: van der Corput's lemma}), is that 
\[
I_{q,u}(h,r)\ll \|\Psi\|_\infty (e^{-u}q^{A-1})^{-1/2} \ll e^{q} q^{1-A}.
\]

\end{proof}

Now we are in the position to prove:

\begin{lemma}
  We have that, as a function of $s\in[0,1]$, the sup-norm $\Vert T_{N}\Vert_{\infty}\ll(\log N)^{-8A}$.
\end{lemma}

\begin{proof}
  Note that 
  \begin{equation}
    \mathcal{E}_{q,u}^{(\mathrm{B})}(s)\ll\frac{1}{N}\sum_{r\asymp e^{u-q}q^{A-1}}\vert\Xi(r)\vert\qquad\mathrm{where}\qquad\Xi(r):=\sum_{k\geq0}\Psi_{q,u}(k,r,s)e(\phi(k,r)-ks).\label{eq: Poisson on E^B}
  \end{equation}
  By Poisson summation, 
  \[
  \Xi(r)=\sum_{h\in\mathbb{Z}}I_{q,u}(h,r).
  \]
  We decompose the right hand side into the contribution $\Xi_{1}(r)$ coming from $\vert h\vert>(4Q)^{A}$, and a contribution $\Xi_{2}(r)$ from the regime $\vert h\vert\leq(4Q)^{A}$. Next, we argue that $\Xi_{1}(r)$ can be disposed off by partial integration. Because $x_{k,r}\leq2N$, we have
  \[
  \omega(x_{k,r})=(\log x_{k,r})^{A}\leq(3Q)^{A}.
  \]
  Note for $\vert h\vert > (4Q)^{A}$, by \eqref{eq: second B 1 deriv comp}, the inequality 
  \[
  \partial_{k}[\phi(k,r)-k(s+h)]\gg h
  \]
  holds true, uniformly in $r$ and $s$. As a result, partial integration yields, for any constant $C>0$, the bound
  \begin{align*}
    I_{q,u}(h,r)\ll\left\Vert k\mapsto\partial_{k}\Psi_{q,u}(k,r,s)\right\Vert _{L^{1}(\mathbb{R})}h^{-C}.
  \end{align*}
  Therefore,
  \begin{align*}
    \Xi_{1}(r)\ll\left\Vert k\mapsto\partial_{k}\Psi_{q,u}(k,r,s)\right\Vert _{L^{1}(\mathbb{R})}\sum_{h\geq(4Q)^{A}}h^{-C}.
  \end{align*}
  Recall that we have $q\geq \frac{1}{m+1} Q$. Thus, taking $C$ to be large and using Lemma \ref{lem: L1 bound}, we deduce that
  \begin{equation}
    \Xi_{1}(r)\ll_{C_{1}}e^{\frac{u}{2}}Q^{-C_{1}}\label{eq: bound on Xi1}
  \end{equation}
  for any constant $C_{1}>0$.  All in all, we have shown so far
  \begin{align*}
    \Xi(r)\ll\Xi_{2}(r)+e^{\frac{u}{2}}Q^{-C_{1}}.
  \end{align*}
  In $\Xi_{2}(r)$ there are $O(Q^{A})$ choices of $h$. By using Lemma \ref{lem: second derivative bound} we conclude 
  \begin{equation}
    \Xi_{2}(r)\ll e^{q}q^{1-A}Q^{A}.\label{eq: bound on Xi2}
  \end{equation}
  By combining (\ref{eq: bound on Xi1}) and (\ref{eq: bound on Xi2}), we deduce from (\ref{eq: Poisson on E^B}) that
  \[
  \left\Vert \mathcal{E}_{q,u}^{(\mathrm{B})}(\cdot)\right\Vert _{\infty}\ll\frac{1}{N}\sum_{r\asymp e^{u-q}q^{A-1}}e^{q}q^{1-A} Q^{A}\ll \frac{1}{N}e^{u} Q^{A}.
  \]
  As a result, 
  \[
  \left\Vert T_{N}(\cdot)\right\Vert_\infty \ll\frac{1}{N}\sum_{(u,q)\in\mathcal{T}(N)}e^{u}Q^{A}\ll\frac{1}{N}\sum_{u\leq\log N-10A\log\log N}e^{u}Q^{A+1}\ll\frac{1}{(\log N)^{10A}}(\log N)^{A+1}\ll\frac{1}{(\log N)^{8A}}.
  \]
  
\end{proof}

\subsection*{Essential regimes}

At this stage, we are ready to apply our stationary phase expansion (Proposition \ref{prop: stationary phase}), and thus effectively apply the $B$-process a second time. Recall that after applying Poisson summation, the phase will be $\psi_{r,h}(t) = \psi(t): = \phi(t,r) - t(h-s)$. Let
\begin{align*}
  W_{q,u}(t):=\frac{\mathfrak{N}_{q}(x_{t,r})\mathfrak{K}_{u}(t)}{\sqrt{t\omega''(x_{t,r})}}\widehat{f}\left(\frac{t}{N}\right)e\left(\psi(t)-  \psi(\mu)-\frac{1}{2}(t-\mu)^{2} \psi^{\prime\prime}(\mu) \right).
\end{align*}
Further, define 
\begin{align*}
  p_{j}(\mu):=c_j \left(\frac{1}{\psi''(\mu)}\right)^{j}W_{q,u}^{(2j)}(\mu),
\end{align*}
where $p_0(\mu) =  e(1/8) W_{q,u}(\mu)$. Note that, by \eqref{eq: size of second deriv in k}, one can bound
\begin{align}\label{p bounds}
  p_{j}(\mu) \ll p_1(\mu) \ll N^{\varepsilon}\frac{1}{\psi^{\prime\prime}(\mu)}\frac{1}{\mu^{1/2}} \frac{\omega^{\prime\prime\prime\prime}(x_{\mu,r}) \left(\partial_t x_{t,r}|_{t=\mu}\right)^2}{\omega^{\prime\prime}(x_{\mu,r})^{3/2}} \ll e^{u/2-q} N^\varepsilon, \qquad j \ge 1.
\end{align}
Hence let
\begin{align*}
  P_{q,u}(h,r,s):=\frac{e(\psi(\mu))}{\sqrt{\psi^{\prime\prime}(\mu)}}\left(p_{0}(\mu)+p_1(\mu)\right),
\end{align*}
and set 
\begin{align*}
  E_{q,u}^{\mathrm{(BB)}}(s):=\frac{e(-1/8)}{N}\sum_{r\geq0}\sum_{h\geq0}P_{q,u}(h,r,s).
\end{align*}
Before proving Proposition \ref{B proc twice} we need the following lemma.

\begin{lemma}\label{lem: average of double bar}
  For any $c\in[0,1]$ and any $M>10$, we have the bound
  \[
  \int_{0}^{1}\,\min (\left\Vert c+s\right\Vert ^{-1},M )\,\mathrm{d}s \leq 2 \log M.
  \]
\end{lemma}
\begin{proof}
  Decomposing into intervals where
  $\left\Vert c+s\right\Vert ^{-1}\leq M$
  as well as intervals 
  where $\left\Vert c+s\right\Vert ^{-1}>M$
  and then using straightforward estimates imply the claimed bound.
\end{proof}

Now we can prove Proposition \ref{B proc twice}.
\begin{proof}[Proof of Proposition \ref{B proc twice}]
  Fix $ s \in [0,1]$ and recall the definition of $I_{q,u}(h,r)$ from \eqref{I def}, then by Poisson summation
  \begin{align*}
    \cE_{q,u}^{\mathrm{(B)}}(s)=\frac{e(-1/8)}{N}\sum_{r\geq0}\sum_{h\in\mathbb{Z}}I_{q,u}(h,r).
  \end{align*}
  Let $[a,b]:=\supp(\mathfrak{K}_u)$, and 
  \begin{align*}
    m_{r}(h):=\min_{k\in[a,b]}\vert \psi_{r,h}^\prime(k)\vert.
  \end{align*}
  We decompose the $h$-summation into three different ranges:
  $$\sum_{h \in \Z} I_{q,u}(h,r) = \mathfrak{C}_1+\mathfrak{C}_2+\mathfrak{C}_3,$$
  where the first contribution, $\mathfrak{C}_{1}(r,s)$ is where $m_{r}(h)=0$, the second contribution, $\mathfrak{C}_{2}(r,s)$ is where $0<m_{r}(k)\leq N^{\varepsilon}$, and the third contribution $\mathfrak{C}_{3}(r,s)$ is where $m_{r}(h)\geq N^{\varepsilon}$.  Integration by parts shows that
  $$\mathfrak{C}_{3}(r,s)\ll N^{-100}.$$

   Next, we handle $\mathfrak{C}_{1}(r,s)$. To this end, we shall apply Proposition \ref{prop: stationary phase}, in whose notation we have 
  \begin{align*}
    \Omega_w & :=e^{u},\qquad
    \Lambda_w  :=e^{q-u/2+\varepsilon},\qquad \Lambda_{\psi}:=e^{u}q^{A-1},\qquad \Omega_{\psi}:=e^{u}.
  \end{align*}
  The decay of the amplitude function was shown in Lemma \ref{lem:amp bounds}, the decay of the phase function follows from a short calculation we omit. Next, since we have disposed of the inessential regimes, we see 
  \[
  Z:=\Omega_{\psi}+\Lambda_w+\Lambda_{\psi}+\Omega_w+1\asymp e^{u}q^{A-1}\asymp N^{1+o(1)}.
  \]
  Further, 
  \[
  \frac{\Omega_{\psi}}{\Lambda_{\psi}^{1/2}}Z^{\frac{\delta}{2}}=\frac{e^{\frac{u}{2}}}{q^{\frac{1}{2}A}}Z^{\frac{\delta}{2}}\asymp\frac{e^{u(\frac{1}{2}+\delta)}}{q^{\frac{1}{2}A+\frac{\delta}{2}(A-1)}}.
  \]
  Hence taking $\delta:=1/11$ is compatible with the assumption (\ref{eq: constraints on X,Y,Z}).
  Thus 
  \[
  \mathfrak{C}_{1}(r,s)=\sum_{h\geq0} P_{q,u}(h,r,s) +  O(N^{-1/11}).
  \]

  Now we bound $\mathfrak{C}_{2}(r,s)$. First note that $\omega(x_{t,r})$ is monotonic in $t$. To see this set the derivative equal to $0$:
  \begin{align*}
    A\frac{\log(x_{t,r})^{A-1}}{x_{t,r}} \partial_t x_{t,r} = A\frac{\log(x_{t,r})^{A-1}}{x_{t,r}} \omega^{\prime}(r/t)(-r/t^2) = 0.
  \end{align*}
  However, since $x_{t,r}\asymp e^q$, this implies $\omega^\prime(r/k) =0$, which is a contradiction. Thus, by van der Corput's lemma (Lemma \ref{lem: van der Corput's lemma}) for the first derivative, and monotonicity, we have 
  \[
  \mathfrak{C}_{2}(r,s)\ll N^{\frac{1}{2}+\varepsilon}\min\left(\left\Vert \omega(x_{a,r})+s\right\Vert ^{-1},N^{\frac{1}{2}+o(1)}\right)
  \]
  where we used \eqref{eq: size of second deriv in k} and the fact that $\partial_t \phi(t,r) = \omega(x_{t,r})$. Notice that
  \[
  \Big\Vert \frac{1}{N}\sum_{r\geq0}\mathfrak{C}_{2}(r,\cdot)
  \Big\Vert _{L^{m}}^{m}
  \ll
  \Big\Vert \frac{1}{N}\sum_{r\geq0}\mathfrak{C}_{2}(r,\cdot)
  \Big\Vert _{\infty}^{m-1}
  \Big\Vert \frac{1}{N}
  \sum_{r\geq0}\mathfrak{C}_{2}(r,\cdot)
  \Big\Vert_{L^{1}}.
  \]
  By (\ref{eq: bound on oscillatory integral second B}) we see 
  \[
  \Big\Vert \frac{1}{N}\sum_{r\geq0}\mathfrak{C}_{2}(r,s)
  \Big\Vert _{\infty}^{m-1}\ll N^{O(\varepsilon)}.
  \]
  Hence it remains to estimate
  \[
  N^{O(\varepsilon)}\sum_{r\asymp e^{u-q}q^{A-1}}\frac{1}{\sqrt{N}}\int_{0}^{1}\,\min\left(\left\Vert \omega(x_{a,r})+s\right\Vert ^{-1},N^{\frac{1}{2}+o(1)}\right)\,\mathrm{d}s.
  \]
  By exploiting Lemma \ref{lem: average of double bar} we see
  \[
  \Big\Vert \frac{1}{N}
  \sum_{r\geq0}\mathfrak{C}_{2}(r,\cdot)
  \Big\Vert _{L^{m}}^m
  \ll
  \sum_{r\asymp e^{u-q}q^{A-1}}
  \frac{N^{o(1)}}{\sqrt{N}}\ll N^{o(1)-\frac{1}{2}}
  \]
  which implies the claim.

  Finally, it remains to show that
  \begin{align*}
    \|\cE^{(BB)}_{q,u}(\cdot) - E^{(BB)}_{q,u}(\cdot)\|_{L^m}^m =O(N^{-1/2 +\varepsilon})
  \end{align*}
  from which we can apply the triangle inequality to conclude Proposition \ref{B proc twice}. For this, recall the bounds \eqref{p bounds}. Since $\cE_{q,u}^{(BB)}$ is simply the term arising from $p_0(\mu)$, we have that 
  \begin{align*}
    \|\cE^{(BB)}_{q,u}(\cdot) - E^{(BB)}_{q,u}(\cdot)\|_{L^m}^m \ll \frac{1}{N}\sum_{r\in \Z} p_1 \ll \frac{e^{u-q/2}N^\varepsilon}{N}.
  \end{align*}
  From here the bound follows from the ranges of $q$ and $u$.

\end{proof}

  Before proceeding, we note that \eqref{EBB def} can be simplified. In particular
  \begin{lemma}\label{lem:Phi mu}
    Given, $h$, $r$, and $s$ as above, we have
    \begin{align*}
      \mu_{h,r,s} = \frac{r}{\omega^\prime(\omega^{-1}(h-s))}, \qquad \Phi(h,r,s) = -r \omega^{-1}(h-s)
    \end{align*}
    and moreover
    \begin{align}\label{kjbkhb}
      \mu \omega^{\prime\prime}(x_{\mu,r}) \cdot (\partial_{\mu\mu}\phi)(\mu,r) =-\frac{r^2}{\mu^2}.
    \end{align}
  \end{lemma}

  \begin{proof}
    Recall $x_{k,r}=\tilde{\omega}\left(\frac{r}{k}\right)$. Now, to compute $\mu$, we have:
    \begin{align*}
      0 & = \partial_\mu \left(\mu \omega(x_{\mu,r})-rx_{\mu,r}\right) -( h-s)\\
        & = \omega(x_{\mu,r}) + \mu \omega^\prime(x_{\mu,r}) \left(\partial_\mu x_{\mu,r}\right) - r \partial_\mu x_{\mu,r} - (h-s).
    \end{align*}
    Consider first
    \begin{align*}
      \partial_\mu x_{\mu,r} &= \partial_\mu\left(\wt{\omega}\left(\frac{r}{\mu}\right)\right)
      = \wt{\omega}^\prime \left(\frac{r}{\mu}\right) \left(-\frac{r}{\mu^2}\right).
    \end{align*}
    Furthermore, since $\mu = \wt{\omega}(\omega^\prime(\mu))$, we may differentiate both sides and then change variables to see
    \begin{align}\label{omegaprimprime}
      \wt{\omega}^\prime(r/\mu) = \frac{1}{\omega^{\prime\prime}(\wt{\omega}(r/\mu))}.
    \end{align}
    Hence
    \begin{align*}
      \partial_\mu (x_{\mu, r}) = - \omega\left(\wt{\omega}\left(\frac{r}{\mu}\right)\right) \frac{r}{\mu^2 \omega^{\prime\prime}(\wt{\omega}(r/\mu))}.
    \end{align*}
    Hence
    \begin{align*}
      0 & = \omega(x_{\mu,r}) - r \left( \omega\left(\wt{\omega}\left(\frac{r}{\mu}\right)\right) \frac{r}{\mu^2 \omega^{\prime\prime}(\wt{\omega}(r/\mu))}\right) + \omega\left(\wt{\omega}\left(\frac{r}{\mu}\right)\right) \frac{r^2}{\mu^2 \omega^{\prime\prime}(\wt{\omega}(r/\mu))} - (h-s)\\
      &= \omega(x_{\mu,r})  - (h-s).
    \end{align*}
    Hence $\omega(\wt{\omega}(r/\mu)) = h-s$. Solving for $\mu$ gives:
    \begin{align*}
      \mu = \frac{r}{\omega^\prime(\omega^{-1}(h-s))}.
    \end{align*}

    Moreover, we can simplify the phase as follows
    \begin{align*}
      \Phi(h,r,s) &= \phi(\mu,r) - (h-s)\mu\\
      &= \mu \omega(x_{\mu,r}) - r x_{\mu,r} - (h-s)\mu\\
      &= \mu \omega(\wt{\omega}(r/\mu)) - r \wt{\omega}(r/\mu) - (h-s)\mu\\    
      &= \frac{r(h-s)}{\omega^\prime(\omega^{-1}(h-s))}  - r \omega^{-1}(h-s) - (h-s)\frac{r}{\omega^\prime(\omega^{-1}(h-s))}\\
      &= - r \omega^{-1}(h-s).
    \end{align*}
    Turning now to \eqref{kjbkhb}, we note that since, by the definition of $\mu$ we have that  $\partial_\mu \phi(\mu,r) = h-s$, and $h-s = \omega (\wt{\omega}(r/\mu))$ we may differentiate both sides of the former to deduce
    \begin{align*}
      \partial_{\mu\mu}\phi(\mu,r) &=  \partial_\mu\left( \omega (\wt{\omega}(r/\mu))\right)\\
      &= \omega^\prime(\wt{\omega}(r/\mu)) \wt{\omega}^\prime(r/\mu) (- r/\mu^2)\\
      &= - (r^2/\mu^3) \wt{\omega}^\prime(r/\mu).
    \end{align*}
    Now using \eqref{omegaprimprime} we conclude that
    \begin{align*}
      \mu \omega^{\prime\prime}(x_{r,\mu}) \cdot (\partial_{\mu\mu}\phi)(\mu,r)
      &= -\mu \omega^{\prime\prime}(\wt{\omega}(r/\mu)) (r^2/\mu^3) \wt{\omega}^\prime(r/\mu)\\
      &= -(r^2/\mu^2)  \omega^{\prime\prime}(\wt{\omega}(r/\mu)) \frac{1}{\omega^{\prime\prime}(\wt{\omega}(r/\mu))} = -\frac{r^2}{\mu^2}.
    \end{align*}
    
\end{proof}

  Applying Lemma \ref{lem:Phi mu} and inserting some definitions allows us to write
  \begin{align} \label{EBB redef}
    \cE^{(BB)}_{q,u}(s)=
    \frac{1}{N}\sum_{r\ge 0}\sum_{h \ge 0}
    \wh{f} \left(\frac{\mu}{N}\right)\mathfrak{K}_u(\mu)\mathfrak{N}_{q}(\wt{\omega}(r/\mu)) \frac{\mu}{r} e(-r\omega^{-1}(h-s)).
  \end{align}

Returning now to the full $L^m$ norm, let $\sigma_i := \sigma(u_i) := \frac{u_i}{\abs{u_i}}$. Proposition \ref{prop:E B triple}, Proposition \ref{B proc twice} and expanding the $m^{th}$-power yields
\begin{equation}\label{eq: cal F B-processed}
  \mathcal{F}(N) = \sum_{\sigma_1,\ldots,\sigma_m\in \{\pm 1\} }
  \sum_{\substack{(u_i,q_i) \in \mathscr{G}(N)\\ u_i>0}}
  \int_{0}^{1} 
  \prod_{\substack{i\leq m\\ \sigma_i >0}} 
  \cE^{\mathrm{(BB)}}_{q_i, u_i}(s)
  \prod_{\substack{i\leq m\\ \sigma_i <0}} 
  \overline{\cE^{\mathrm{(BB)}}_{q_i, u_i}(s)}
  \rd s
  + O(N^{-\varepsilon/2}).
\end{equation}
To simplify this expression, for a fixed $\vect{u}$ and $\vect{q}$, and $\vect{\mu}=(\mu_1,\ldots,\mu_m)$ let $\mathfrak{K}_{\vect{u}}(\vect{\mu}):= \prod_{i\leq m}\mathfrak{K}_{u_i}(\mu_i) $. The functions $ \mathfrak{N}_{\vect{q}}\left(\vect{\mu},s\right)$ and $\wh{f}(\vect{\mu}/N)$ are defined similarly. Aside from the error term, the right hand side of \eqref{eq: cal F B-processed} splits into a sum over
    \begin{align*}
      \cF_{\vect{q},\vect{u}} := \frac{1}{N^m} 
      \sum_{\vect{r} \in \Z^m}\frac{1}{r_1r_2\cdots r_m}
      \int_0^1\sum_{\vect{h}\in \Z^m }
      \mathfrak{K}_{\vect{u}}(\vect{\mu}) 
      \mathfrak{N}_{\vect{q}}\left(\vect{\mu},s\right) 
      A_{\vect{h},\vect{r}}(s) 
      e\left( \varphi_{\vect{h},\vect{r}}(s)\right) \mathrm{d}s 
    \end{align*}
    where the phase function is given by
  \begin{align*}
    \varphi_{\vect{h},\vect{r}}(s) := -(r_1\omega^{-1}(h_1-s)+r_2\omega^{-1}(h_2-s) +  \dots  + r_m\omega^{-1}(h_m-s)),
  \end{align*}
  and where
  \begin{align*}
    A_{\vect{h},\vect{r}}(s) : =  \wh{f}\left(\frac{ \vect{\mu}}{N}\right) \mu_1\mu_2 \dots \mu_m.
  \end{align*}
   Now we distinguish between two cases. First, the set of all $(\vect{r},\vect{h})$  where the phase $\varphi_{\vect{h},\vect{r}}(s)$ vanishes identically, which we call the \emph{diagonal}; and its complement, the \emph{off-diagonal}. Let
  \begin{align*}
    \mathscr{A} : = \{ (\vect{r}, \vect{h}) \in \N\times \N : 
    \varphi_{\vect{h},\vect{r}}(s) = 0, \forall s \in [0,1] \},
  \end{align*}
  and let
  \begin{align*}
    \eta(\vect{r},\vect{h}):=\begin{cases}
    1 & \mbox{ if } (\vect{r},\vect{h}) \not\in \mathscr{A} \\
    0 & \mbox{ if } (\vect{r},\vect{h}) \in \mathscr{A}.
    \end{cases}
  \end{align*}
 The diagonal, as we show, contributes the main term, while
 the off-diagonal contribution is negligible (see section \ref{s:Off-diagonal}).

  \section{Extracting the Diagonal}
  \label{s:Diagonal}
  First, we establish an asymptotic for  the diagonal. The below sums range over $\vect{q} \in [2Q]^m$, $\vect{u}\in [-U,U]$, and $\vect{r},\vect{h}\in \Z$.  Let
    \begin{align*}
      \cD_N
      &= \frac{1}{N^m}\sum_{\vect{q},\vect{u},\vect{r},\vect{h} }(1-\eta(\vect{r},\vect{h}))\frac{1}{r_1r_2\cdots r_m}  \int_0^1 \mathfrak{K}_{\vect{u}}(\vect{\mu}) \mathfrak{N}_{\vect{q}}\left(\vect{\mu},s\right) A_{\vect{h},\vect{r}}(s)  \mathrm{d}s.
    \end{align*}
    With that, the following lemma establishes the main asymptotic needed to prove Lemma \ref{lem:MP = KP non-isolating} (and thus Theorem \ref{thm:correlations}).

    \begin{lemma}\label{lem:diag}
      We have
      \begin{align}\label{diag}
       \lim_{N\rightarrow \infty} \cD_N = 
       \sum_{\cP\in \mathscr{P}_m} 
       \mathbf{E}(f^{\abs{P_1}})\cdots \mathbf{E}(f^{\abs{P_d}}).
      \end{align}
      where the sum is over all non-isolating partitions of $[m]$, 
      which we denote $\cP = (P_1, \dots, P_d)$.
    \end{lemma}

    \begin{proof}
      Since the Fourier transform $\wh{f}$ is assumed to have compact support, we can evaluate the sum over $\vect{u}$ and eliminate the factors $\mathfrak{K}_{\vect{u}}$. Hence
    \begin{align*}
      \cD_N 
      &= \frac{1}{N^m}\sum_{\vect{q},\vect{r},\vect{h} } \one(\abs{\mu_i} >0)(1-\eta(\vect{r},\vect{h}))\frac{1}{r_1r_2\dots r_m}   \int_0^1 \mathfrak{N}_{\vect{q}}\left(\vect{\mu},s \right) A_{\vect{h},\vect{r}}(s)  \mathrm{d}s,
    \end{align*}
    here the indicator function takes care of the fact that we extracted the contribution when $k_i=0$.

    The condition that the phase is zero, is equivalent to a condition on $\vect{h}$ and $\vect{r}$. Specifically, this happens in the following situation: let $\cP$ be a non-isolating partition of $[m]$, we say a vector $(\vect{r},\vect{h})$ is $\cP$\emph{-adjusted} if for every $P \in \cP$ we have: $h_i = h_j$ for all $i, j \in P$, and $\sum_{i\in P} r_i = 0$. The diagonal is restricted to $\cP$-adjusted vectors. Now 
    \begin{align*}
      \chi_{\cP,1}(\vect{r})
      :=
      \begin{cases}
        1 & \mbox{ if $\sum_{i\in P} r_i =0$ for each $P \in \cP$}\\
        0 & \mbox{ otherwise,}
      \end{cases},
      \qquad
      \chi_{\cP,2}(\vect{h})
      :=
      \begin{cases}
        1 & \mbox{ if $h_i =h_j$ for all $i,j \in P\in \cP$}\\
        0 & \mbox{ otherwise,}
      \end{cases}
    \end{align*}
    here $\chi_{\cP,1}(\vect{r})\chi_{\cP,2}(\vect{h})$ encodes the condition that $(\vect{r},\vect{h})$ is $\cP$-adjusted. Thus, we may write
    \begin{align*}
      \cD_N 
      &= \frac{1}{N^m}\sum_{\cP\in \mathscr{P}_m} \sum_{\vect{q},\vect{r}, \vect{h} }\chi_{\cP,1}(\vect{r})\chi_{\cP,2}(\vect{h})\frac{1}{r_1r_2\cdots r_m}\left( \int_0^1 \mathfrak{N}_{\vect{q}}\left(\vect{\mu}, s\right)\wh{f}\left(\frac{ \vect{\mu}}{N}\right) \mu_1\mu_2\cdots \mu_m  \mathrm{d}s\right) +o(1).
    \end{align*}
    Inserting the definition of $\mu_i$ then gives 
    \begin{align*}
      \cD_N 
      &= \frac{1}{N^m}\sum_{\cP\in \mathscr{P}_m}\sum_{\vect{q},\vect{r}, \vect{h} } \chi_{\cP,1}(\vect{r})\chi_{\cP,2}(\vect{h})  \int_0^1 \mathfrak{N}_{\vect{q}}\left(\vect{\mu},s\right)\wh{f}\left(\frac{\vect{\mu}}{N}\right)   \prod_{i=1}^m\left(\frac{1}{\omega^\prime(\omega^{-1}(h_i-s))}\right)  \mathrm{d}s +o(1).
    \end{align*}
Now note that the $\vect{r}$ variable only appears in $\wh{f}\left(\vect{\mu}/N\right)$, that is
    \begin{align}\label{D intermediate}
      \begin{aligned}
      \cD_N 
      &= \frac{1}{N^m}\sum_{\cP\in \mathscr{P}_m}
      \sum_{P \in \cP}
      \sum_{\vect{q}, h }  
       \int_0^1 \mathfrak{N}_{\vect{q},P}\left(h\right)
       \left(\frac{1}{\omega^\prime(\omega^{-1}(h))}\right)^{\abs{P}}
      \sum_{\substack{\vect{r}\in \Z^{\abs{P}}\\r_i \neq 0}}
      \chi(\vect{r}) \wh{f}\left(\frac{1}{N\omega^\prime(\omega^{-1}(h))} 
      \vect{r}\right)     \mathrm{d}s(1+ o(1)),
      \end{aligned}
    \end{align}
    where $\chi(\vect{r})$ is $1$ if $\sum_{i=1}^{\abs{P}} r_i = 0$ 
    and where $\mathfrak{N}_{\vect{q},P}(h) = \prod_{i\in P}\mathfrak{N}_{q_i} (\mu_i, s ) $. We can apply Euler's summation formula (\cite[Theorem 3.1]{Apostol1976}) to conclude that
    \begin{align*}
      \sum_{\substack{\vect{r}\in \Z^{\abs{P}}\\r_i \neq 0}}\chi(\vect{r}) \wh{f}\left(\frac{1}{N\omega^\prime(\omega^{-1}(h))} 
      \vect{r}\right)   
      =  
      \int_{\R^{\abs{P}}} \chi(\vect{x})\wh{f}\left(\frac{1}{N\omega^\prime(\omega^{-1}(h))} \vect{x}\right) \mathrm{d}\vect{x}\left(1+ o(1) \right).
    \end{align*}
    Changing variables then yields
    \begin{align*}
      & \int_{\R^{\abs{P}}} \chi(\vect{x})\wh{f}\left(\frac{1}{N\omega^\prime(\omega^{-1}(h))} \vect{x}\right)\mathrm{d}\vect{x}
      =N^{\abs{P}-1}\omega^\prime(\omega^{-1}(h))^{\abs{P}-1}  \int_{\R^{\abs{P}}} \chi_{\cP}(\vect{x}) \wh{f}\left(\vect{x}\right)\mathrm{d}\vect{x} \left(1+ O\left(N^{-\theta}\right)\right),
    \end{align*}
    note that $\chi(\vect{x})$ fixes $x_{\abs{P}} = - \sum_{i=1}^{\abs{P}-1} x_i$.  Plugging this into our \eqref{D intermediate} gives
    \begin{align*}
      \cD_N 
      &= \frac{1}{N^d}\sum_{\cP\in \mathscr{P}_m}\sum_{P \in \cP}
      \sum_{\vect{q}, h }    \mathfrak{N}_{\vect{q},P}\left(h\right)\omega^\prime(\omega^{-1}(h)) 
      \int_{\R^{\abs{P}-1}} \wh{f}(x_1, \dots, x_{\abs{P}-1}, -\vect{x}\cdot \vect{1}) \rd \vect{x} (1+ o(1)).
    \end{align*}
    Next, we may apply Euler's summation formula and a change of variables to conclude that
    \begin{align*}
      \cD_N 
      &= \sum_{\cP\in \mathscr{P}_m}\sum_{P \in \cP}
      \left( \int_{\R^{\abs{P}-1}} \wh{f}(x_1, \dots, x_{\abs{P}-1}, -\vect{x}\cdot \vect{1}) \rd \vect{x} \right)(1+ o(1)).
    \end{align*}
    From there we apply Fourier analysis as in \cite[Proof of Lemma 5.1]{LutskoTechnau2021} to conclude \eqref{diag}.

    \end{proof}

  \section{Bounding the Off-Diagonal}
  \label{s:Off-diagonal}

  Recall the off-diagonal contribution is given by
  \begin{align*}
      \cO_N
      &= \sum_{\vect{q},\vect{u}}\frac{1}{N^m}\sum_{\vect{r},\vect{h} }\eta(\vect{r},\vect{h})\int_0^1\frac{\mu_1\mu_2\cdots \mu_m}{r_1r_2\cdots r_m} \wh{f}\left(\frac{\vect{\mu}}{N}\right)   \mathfrak{K}_{\vect{u}}(\vect{\mu}) \mathfrak{N}_{\vect{q}}\left(\vect{\mu},s\right)e(\Phi(\vect{h},\vect{r},s))  \mathrm{d}s.
    \end{align*}
  where $r_i \asymp e^{u_i-q_i}q_i^{A-1}$, the variable $h_i \asymp q^{A}$. Finally the phase function 
  \begin{align*}
    \Phi(\vect{h},\vect{r},s) = -\sum_{i=1}^m r_i \exp((h_i-s)^{1/A}).
  \end{align*}
  
  If we were to bound the oscillatory integral trivially, we would achieve the bound $\cO_N \ll (\log N)^{(A+1)m}$. Therefore all that is needed is a small power saving, for which we can exploit the oscillatory integral
  \begin{align*}
    I(\vect{h},\vect{r}) := \int_0^1 \cA_{\vect{h},\vect{r}}(s) e(\Phi(\vect{h},\vect{r},s))  \mathrm{d}s
  \end{align*}
  where
  \begin{align*}
    \cA_{\vect{h},\vect{r}}(s) : = \frac{\mu_1\mu_2\cdots \mu_m}{r_1r_2\cdots r_m} \wh{f}\left(\frac{\vect{\mu}}{N}\right)   \mathfrak{K}_{\vect{u}}(\vect{\mu}) \mathfrak{N}_{\vect{q}}\left(\vect{\mu},s\right).
  \end{align*}
  While bounding this integral is more involved in the present setting, we can nevertheless use the proof in \cite[Section 6]{LutskoTechnau2021} as a guide. In Proposition \ref{prop:int s}, we achieve a power-saving, for this reason we can ignore the sums over $\vect{q}$ and $\vect{u}$ which give a logarithmic number of choices.

  Since we are working on the off-diagonal we may write the phase as
  \begin{align}\label{Phi}
    \Phi(\vect{h},\vect{r},s) = \sum_{i=1}^l r_i \exp((h_i-s)^{1/A}) - \sum_{i=l}^L r_i \exp((h_i-s)^{1/A}),
  \end{align}
  where we may now assume that $r_i >0$, the $h_i$ are pairwise distinct, and $L< m $. We restrict attention
  to the case $L = m$ (this is also the most difficult case
  and the other cases can be done analogously).

  \begin{proposition}\label{prop:int s}
    Let $\Phi$ be as above, then for any $\vep>0$ we have
    \begin{align*}
      I(\vect{h},\vect{r}) \ll N^{\varepsilon}\mathfrak{K}_{\vect{u}}(\vect{\mu}_0) \mathfrak{N}_{\vect{q}}(\vect{\mu}_0,0) \frac{e^{u_1 + \dots + u_m}}{r_1 \cdots r_m} N^{- 1/m +\vep} 
    \end{align*}
    as $N \to \infty$, where $\mu_{0,i} =\frac{r_i}{\omega^\prime(\omega^{-1}(h_i))}$. The implied constants are independent of $\vect{h}$ and $\vect{r}$ provided $\eta_{\vect{r}}(\vect{h}) \neq 0$.  
  \end{proposition}

  \begin{proof}
  
  As in \cite{LutskoTechnau2021} we shall prove Proposition \ref{prop:int s} by showing that one of the first $m$ derivatives of $\Phi$ is small. Then we can apply van der Corput's lemma to the integral and achieve the necessary bound. However, since the phase function is a sum of exponentials (as oppose to a sum of monomials as it was in our previous work), achieving these bounds is significantly more involved than in \cite{LutskoTechnau2021}.

The $j^{th}$ derivative is given by (we will send $s \mapsto -s$ to avoid having to deal with minus signs at the moment)
  \begin{align*}
    D_j  &= \sum_{i=1}^m r_i \exp((h_i+s)^{1/A}) \left\{ A^{-j} (h_i+s)^{j/A-j} + c_{j,1}(h_i+s)^{(j-1)/A-j} + \dots + c_{j,j-1}(h_i+s)^{(1/A-j)} \right\}\\
    &= : \sum_{i=1}^m b_i P_j(h_i) 
  \end{align*}
  where the $c_{j,k}$ depend only on $A$ and $j$, where $b_i := r_i \exp((h_i+s)^{1/A})$.

  In matrix form, let $\vect{D} :=(D_1, \dots, D_m)$ denote the vector of the first $m$ derivatives, and let $\vect{b}:= (b_1, \dots, b_m)$. Then
  \begin{align*}
    &\vect{D} = \vect{b} M,
    \qquad \mbox{ where } (M)_{ij}:=P_j(h_i).
  \end{align*}

  To prove Proposition \ref{prop:int s} we will lower bound the determinant of $M$. Thus we will show that $M$ is invertible, and hence we will be able to lower bound the $\ell^2$-norm of $\vect{D}$. For this, consider the $j^{th}$ row of $M$
  \begin{align*}
    (M)_j = (P_j(h_1), \dots, P_j(h_m)).
  \end{align*}
  We can write $P_j(h_1) : = \sum_{k=0}^{j-1} c_{j,k} (h_i+s)^{t_k/A-j}$, where $t_k:= j-k$. Since the determinant is multilinear in the rows, we can decompose the determinant of $M$ as
  \begin{align}
    \det(M) = \sum_{\vect{t}\in \cT} c_{\vect{t}} \det((h_i+s)^{t_j/A-j})_{i,j\le m})
  \end{align}
  where $c_{\vect{t}}$ are constants depending only on $\vect{t}$ and the sum over $\vect{t}$ ranges over the set
  \begin{align*}
    \cT : = \{\vect{t} \in \N^m :\ t_j \le j, \ \forall j \in [1,m]   \}.
  \end{align*}
  Let $X_{\vect{t}}:= ((h_i+s)^{t_j/A-j})_{i,j\le m}$. We claim that $\det(M) = c_{\vect{t}_M}\det(X_{\vect{t}_M})(1+O(\max_i (h_i^{-1/A})))$ as $N \to \infty$, where $\vect{t}_M := (1,2,\dots, m)$.

  To establish this claim, we appeal to the work of Khare and Tao, see Lemma \ref{lem:KT}. Namely, let $\vect{H} := (h_1+s, \dots , h_m +s)$ with $h_1 > h_2> \dots$ let $\vect{T}(\vect{t}):= (t_1/A-1, \dots, t_m/A-m)$. Then we can write
  \begin{align*}
    X_{\vect{t}} := \vect{H}^{\circ \vect{T}(\vect{t})}.
  \end{align*}
  Now invoking Lemma \ref{lem:KT} we have
  \begin{align*}
    \det(X_{\vect{t}}) \asymp V(\vect{H}) \vect{H}^{\vect{T}(\vect{t})-\vect{n}_{\text{min}}}.
  \end{align*}
  Note that we may need to interchange the rows and/or columns of $X_{\vect{t}}$ to guarantee that the conditions of Lemma \ref{lem:KT} are met. However this will only change the sign of the determinant and thus won't affect the magnitude.

  Now, fix $\vect{t} \in \cT$ such that $\vect{t} \neq \vect{t}_M$ and compare
  \begin{align*}
    \abs{\det(X_{\vect{t}_M})} - \abs{\det(X_{\vect{t}})} \ge  \abs{V(\vect{H})}\left( \abs{\vect{H}^{\vect{T}(\vect{t}_M)-\vect{n}_{\text{min}}}}-\abs{\vect{H}^{\vect{T}(\vect{t})-\vect{n}_{\text{min}}}}\right).
  \end{align*}
  Since $\vect{t}_M \neq \vect{t}$ we conclude that all coordinates $t_i \le (\vect{t}_M)_i$ and there exists a $k$ such that $t_k < (\vect{t}_M)_k$. Therefore
  \begin{align*}
    \abs{\det(X_{\vect{t}_M})} - \abs{\det(X_{\vect{t}})} =  \abs{V(\vect{H})\vect{H}^{\vect{T}(\vect{t}_M)-\vect{n}_{\text{min}}}}(1+ O(\max_i(h_i^{-1/A})) = \abs{\det(X_{\vect{t}_M})}(1+ O(\max_i(h_i^{-1/A})).
  \end{align*}
  This proves our claim. 

  Hence
  \begin{align}\label{detM bound}
    \begin{aligned}
    \abs{\det(M)} &= \abs{c_{\vect{t}_M} \det(X_{\vect{t}_M})}(1+O(\max_i(h_i^{-1/A}))\\
    &=  \abs{c_{\vect{t}_M}V(\vect{H}) \vect{H}^{\vect{T}(\vect{t}_M)-\vect{n}_{\text{min}}}}(1+O(\max_i(h_i^{-1/A}))\\
    &= \abs{c_{\vect{t}_M}}\left(\prod_{j=1}^m (h_j+s)^{j/A-2j+1 }\right)\prod_{1\le i<j\le m} (h_i-h_j)(1+O(\max_i(h_i^{-1/A}))
    \end{aligned}
  \end{align}
  which is clearly larger than $0$ (since $h_i-h_j >1$ and $s>-1$).

  Hence $M$ is invertible, and we conclude that
  \begin{align*}
    \vect{D} M^{-1} &= \vect{b},\\ 
    \|\vect{D}\|_{2}\| M^{-1}\|_{\text{spec}} &\ge \|\vect{b}\|_2,\\ 
    \|\vect{D}\|_{2} &\ge \frac{\|\vect{b}\|_2}{\| M^{-1}\|_{\text{spec}}},\\ 
  \end{align*}
  where $\|\cdot\|_{\text{spec}}$ denotes the spectral norm with respect to the $\ell^2$ vector norm. Recall that $\|M^{-1}\|_{\text{spec}}$ is simply the largest eigenvalue of $M^{-1}$. Hence $\det(M^{-1})^{1/m} \le \|M^{-1}\|_{\text{spec}}$.

  We can bound the spectral norm by the maximum norm
  \begin{align*}
    \|M^{-1}\|_{\text{spec}} \ll \max_{i,j} (M^{-1})_{i,j}
  \end{align*}
  However each entry of $M^{-1}$ is equal to $\frac{1}{\det(M)}$ times a cofactor of $M$, by Cramer's rule. This, together with the size of the $h_i$ is enough to show that
  \begin{align*} 
    \|\vect{D}\|_{2} &\gg \|\vect{b}\|_2 \log(N)^{-1000m}.
  \end{align*}
  Now using the bounds on $\vect{b}$ (which come from the essential ranges of $h_i$ and $r_i$) we conclude
  \begin{align*}
    \|\vect{D}\|_2 \gg N^{1-\vep}.
  \end{align*}
  From here we can apply the localized van der Corput's lemma {\cite[Lemma 3.3]{TechnauYesha2020}} as we did in \cite{LutskoTechnau2021} to conclude Proposition \ref{prop:int s}.

  \end{proof}

\section{Proof of Lemma \ref{lem:MP = KP non-isolating}}

Thanks to the preceding argument, we conclude that
\begin{align*}
  \lim_{N \to \infty}\cK_{m}(N) &= \sum_{\cP\in \mathscr{P}_m}  \expect{f^{\abs{P_1}}}\cdots \expect{f^{\abs{P_d}}}   + \lim_{N \to \infty} \cO_N.
\end{align*}

Moreover, the off-diagonal term can be bounded using Proposition \ref{prop:int s} as follows:
\begin{align*}
  \cO_N &= \frac{1}{N^m}\sum_{\vect{q},\vect{u}}\sum_{\vect{r}, \vect{h}} \eta(\vect{r,},\vect{h}) I(\vect{h},\vect{r})\\
        &\ll \frac{1}{N^m}\sum_{\vect{q},\vect{u}}\sum_{\vect{r}, \vect{h}}\mathfrak{K}_{\vect{u}}(\vect{\mu}_0) \mathfrak{N}_{\vect{q}}(\vect{\mu}_0,0) \frac{e^{u_1 + \dots + u_m}}{r_1 \cdots r_m} \max_{i \le m} e^{-u_i/m}N^{\varepsilon}.
\end{align*}
Note that we are summing over reciprocals of $r_i$ and recall that the $h_i$ have size $q_i^{A}$, thus, we may evaluate the sums over $\vect{h}$ and $\vect{r}$ and gain at most a logarithmic factor (which can be absorbed into the $\varepsilon$). Thus
\begin{align*}
  \cO_N &\ll \frac{1}{N^m}\sum_{\vect{q},\vect{u}} e^{u_1 + \dots + u_m} \max_{i \le m} e^{-u_i/m}N^{\varepsilon}.
\end{align*}
Likewise there are logarithmically many $\vect{q}$ and $\vect{u}$. Thus maximizing the upper bound, we arrive at
\begin{align*}
  \cO_N &\ll N^{-1/m +\varepsilon},
\end{align*}
this concludes our proof of Lemma \ref{lem:MP = KP non-isolating}. From there, Theorem \ref{thm:main} and Theorem \ref{thm:correlations} follow from the argument in Section \ref{s:Completion}.

   \small 
   \section*{Acknowledgements}
   NT was supported by a Schr\"{o}dinger Fellowship of the Austrian Science Fund (FWF): project J 4464-N. We thank Apoorva Khare, Jens Marklof and Zeev Rudnick for comments on a previous version of the paper. 
   Furthermore, we are grateful to the anonymous referee
   for a careful reading and comments that helped
   remove inaccuracies from an earlier version 
   of this manuscript.
   
\section*{Appendix: 
A sublacunary Partition of Unity}
Choose large integers $N,Q\geq 10$,
with $e^{Q}\leq N< e^{Q+1}$.
In this section, we justify 
that the functions 
$\mathfrak{N}_q$, where $0\leq q \leq 2Q-1$, 
from Section \ref{subsec dyadic}
exist. 
We assume in the following that $Q$
is an odd integer. The case that $Q$
is even can be handled very similarly.
We start from the disjoint decomposition 
$$
[1,N] = [1,e^{Q-1})\cup
[e^{Q-1},N].
$$
To construct a 
suitable partition 
of unity for the first 
and the second set we require somewhat
different treatment.
However the main idea is rather simple,
and the mechanics are the same for both.
We begin with the former.  We cut $[1,e^{Q-1})$
into a union of 
the dyadic intervals $[e^{q},e^{q+1})$.
Then we smooth out the corners
of their indicators
functions in the following way.
We smooth every 
second indicator
function (changing their support only marginally),
and adjust 
the remaining ones so that 
neighboring functions always sum up to one. 
Further, to decompose $[e^{Q-1},N]$
in a sublacunary manner, we use a 
similar strategy.
This time we cut the interval into
a union of sublacunary intervals, 
smoothing the ones with odd indices
and then adjusting the corners of the even
ones appropriately. 

To execute this plan,
fix a smooth function $\beta: 
\mathbb{R}\rightarrow \mathbb{R}_{\geq 0}$
which is supported in a compact interval.
E.g.
\[
\beta (x):=\begin{cases}
\exp\big(-\frac{1}{1-\vert x \vert^2}\big) & \mathrm{for}\,\,0\leq\vert x\vert<1\\
0 & \mathrm{otherwise}
\end{cases}
\]
is a viable choice.
By translating (if need be) and scaling, 
we modify $\beta$ to obtain a function 
$B: \mathbb{R}\rightarrow 
\mathbb{R}_{\geq 0}$
which is smooth, supported in 
$[-1/100, 1/100]$, 
and $L^1$-normalised, that is 
$$ 
\int_{\mathbb{R}} B(x) \mathrm{d}x =1.
$$
Denote the convolution of 
functions $g_1,g_2: \mathbb{R}\rightarrow 
\mathbb{R}$
by 
$$
(g_1*g_2)(x):=
\int_{\mathbb{R}} g_1(x-y) g_2(y) 
\mathrm{d}y.
$$
Clearly,
if $g_1$ is $j$ times continuously 
differentiable, then 
$g_1*g_2$
is $j$ times continuously 
differentiable and
\begin{equation}\label{eq conv}
  (g_1*g_2)^{(j)} = g_1^{(j)}*g_2.  
\end{equation}
\subsection*{The regime $0\leq q< Q$}
Denote the indicator function 
of an interval
$I$ by $\mathbf{1}_{I}$.
The smooth function $B_{y}(x):=e^{-{y}} B(e^{-y}x)$
is important in what follows,
and in particular that it is supported 
in the interval
$[-e^{y}/100,e^{y}/100]$.
For even index $q=2t \in \{0,\ldots,Q-1\}$, we let
$$
\mathfrak{N}_{2t}(x):= 
(\mathbf{1}_{[e^{q},e^{q+1})}*B_q)(x).
$$
These functions
inherit the smoothness of $B$.
By \eqref{eq conv}, we have
\begin{equation}\label{eq even decay}
   \Vert \mathfrak{N}_{2t}^{(j)}\Vert_\infty \ll e^{-qj} 
\end{equation}
for each integer $j\geq 0$. (The implied constant 
is allowed to depend on the choice of $B$,
which is however
fixed throughout and therefore 
this dependency is suppressed.)
Next, for every odd index 
$q=2t+1 \in \{0,\ldots,Q-1\}$
we define
$$
\mathfrak{N}_{2t+1}(x):=
\begin{cases}
0, & \mathrm{if}\,\,
x <0.98 \cdot e^{2t},\\
1-\mathfrak{N}_{2t}(x), 
& \mathrm{if}\,\, 0.98 \cdot e^{2t}\leq x<
1.02 \cdot e^{2t},\\
1, & \mathrm{if}\,\, 1.02 \cdot e^{2t}\leq x < 
0.98 \cdot e^{2t+2},\\
1-\mathfrak{N}_{2t+2}(x), & \mathrm{if}\,\,
0.98 \cdot e^{2t+2} \leq  x < 1.02 \cdot e^{2t+2},\\
0, & \mathrm{if}\,\,e^{2t+2}\leq x.
\end{cases}
$$
Since $\mathfrak{N}_{2t+1}$
is given, piece-wise, by concatenating five smooth functions,
we can infer that $\mathfrak{N}_{2t+1}$
is a smooth function
as soon as we establish
that the four relevant boundary points 
do not cause issues.
A quick inspection
reveals that $\mathfrak{N}_{q}$,
for any $q<Q$,
is supported in 
$[0.98 \cdot e^{q}, 1.02 \cdot e^{q+1}]
\subseteq [e^q/2,3 e^q)$. Hence,
$\mathfrak{N}_{2t+1}$ is smooth. Further,
$\mathfrak{N}_{2t+1}$ is
monotonically increasing 
(though not always strictly) 
on the interval 
$[0.98 \cdot e^{q}, 1.02 \cdot e^{q}]$,
equal to one on 
$[1.02 \cdot e^{q}, 0.98\cdot e^{q+1}]$,
and monotonically decreasing 
(not always strictly) on
$[0.98 \cdot e^{q+1}, 1.02\cdot e^{q+1}]$.
The function $\mathfrak{N}_{2t}$
has analogous monotonicity properties.
Hence, $\mathfrak{N}_{q}'$ 
also has exactly one sign change.
By construction, 
$$
\mathfrak{N}_{q}(x) + \mathfrak{N}_{q+1}(x)
= 1
$$
whenever $e^q \leq x \leq e^{q+1}$,
irrespective
of the parity of $q$.
Moreover, 
\eqref{eq even decay} implies that 
\begin{equation}\label{eq odd decay}
   \Vert \mathfrak{N}_{2t+1}^{(j)}\Vert_\infty \ll e^{-qj}. 
\end{equation}
Consequently, the family of functions 
$\{ \mathfrak{N}_{q}: 0\leq q <Q\}$ satisfies \eqref{N deriv}.
\subsection*{The regime $Q\leq q<2Q$}
Put $\mathcal{Q}:= 2.8 \cdot e^{Q}/Q$.
To render the subsequent formulas
more transparent, define
$a_t:=e^{Q} +(t-1/2)
\mathcal{Q}$.
Observe that
$$
[e^{Q-1},N) = 
[e^{Q-1},e^Q) \cup  [e^{Q},N) \subset
[e^{Q-1},e^Q)  \cup 
\bigcup_{0 \leq  t\leq T+1}
\mathcal{I}_t 
\quad \mathrm{where} \quad 
\mathcal{I}_t:=[a_t, 
a_t +\mathcal{Q})
$$
where $T$ is the unique integer 
so that $N \in \mathcal{I}_{T}$. 
In fact, due to $e/2.8<0.971$, 
we see that $0\leq T \leq 0.98 Q$.
Informally,
our next step is to
again smooth out ever 
other partition interval
in the aforementioned cover.
However, this time
we need to be mindful 
at the end points of the partition
which, necessarily,
play a distinguished role.
The end points will be handled
at the end of this paragraph,
and we first deal with 
the bulk of the intervals. 
Writing any even index $q\in \{Q+1,Q+T+1\}$
in the form $q=Q+2t+1$, with $0\leq t \leq T$,
we define
$$
\mathfrak{N}_{Q+2t+1}(x):= 
(\mathbf{1}_{\mathcal{I}_{2t+1}}*B_\mathcal{Q})(x).
$$
Next, writing any odd index $q\in \{Q+2,Q+T+1\}$
in the form $q=Q+2t$, with $1\leq t \leq T$,
we put
$$
\mathfrak{N}_{Q+2t}(x):=
\begin{cases}
0, & \mathrm{if}\,\,
x <a_{2t-2} - 0.02\cdot \mathcal{Q},\\
1-\mathfrak{N}_{Q+2t-2}(x), 
& \mathrm{if}\,\, a_{2t-2} - 0.02\cdot \mathcal{Q} \leq x<
a_{2t-2} + 0.02\cdot \mathcal{Q},\\
1, & \mathrm{if}\,\, 
a_{2t-2} + 0.02\cdot \mathcal{Q} \leq x < 
a_{2t+2} - 0.02\cdot \mathcal{Q},\\
1-\mathfrak{N}_{Q+2t+2}(x) & \mathrm{if}\,\,
a_{2t+2} - 0.02\cdot \mathcal{Q} \leq  x 
< a_{2t+2} + 0.02\cdot  \mathcal{Q},\\
0, & \mathrm{if}\,\,a_{2t+2} + 0.02\cdot \mathcal{Q}\leq x.
\end{cases}
$$
The function $\mathfrak{N}_{Q}$
plays a distinguished role
as it is the in between 
the dyadic partition 
of unity given provided by the functions
$\{\mathfrak{N}_{q}:0\leq q<Q\}$
and the sub-lacunary one furnished by
$\{\mathfrak{N}_{q}:Q<q<2Q-1\}$.
To smoothly link these two, 
we let
$$
\mathfrak{N}_{Q}(x):=
\begin{cases}
0, & \mathrm{if}\,\,
x < 0.98 e^{Q},\\
1-\mathfrak{N}_{Q-1}(x), 
& \mathrm{if}\,\, 0.98 e^{Q} \leq x<
0.98 e^{Q},\\
1, & \mathrm{if}\,\, 
a_{2t-1} + 00.2\cdot \mathcal{Q} \leq x < 
a_{2t+1} - 00.2\cdot \mathcal{Q},\\
1-\mathfrak{N}_{Q+1}(x) & \mathrm{if}\,\,
a_{2t+1} - 00.2\cdot \mathcal{Q} \leq  x 
< a_{2t+1} + 00.2\cdot \mathcal{Q},\\
0, & \mathrm{if}\,\,a_{2t+1} + 00.2\cdot \mathcal{Q}\leq x.
\end{cases}
$$
To formally complete the construction, we put
$$
\mathfrak{N}_{q}(x):=0,
\quad
\mathrm{for}
\quad q\in \{Q+T+2,\ldots,2Q-1\}.
$$
By arguing very much as in the regime 
$0\leq q< Q$, one can check that 
$\{ \mathfrak{N}_{q}: Q\leq q <2Q\}$
has the required properties and in particular 
is so that \eqref{N deriv} holds true.
  \bibliographystyle{alpha}
  \bibliography{biblio}

    \hrulefill

    \vspace{4mm}
     \noindent Department of Mathematics, Rutgers University, Hill Center - Busch Campus, 110 Frelinghuysen Road, Piscataway, NJ 08854-8019, USA. \emph{E-mail: \textbf{christopher.lutsko@math.uzh.ch}}

\vspace{4mm}
     \noindent
Department of Mathematics,
California Institute of Technology,
1200 E California Blvd.,
Pasadena, CA 91125, USA
\emph{E-mail: 
\textbf{technau@mpim-bonn.mpg.de, ntechnau@caltech.edu}}

\end{document}